\documentclass{amsart}
\usepackage{amssymb}
\usepackage{amsmath}
\usepackage{amsthm}
\usepackage{biblatex}
\usepackage{bm}
\usepackage{enumerate}
\usepackage{geometry}
\usepackage{graphicx}
\usepackage{mathabx}
\usepackage{mathrsfs}
\usepackage{microtype}
\usepackage{tikz-cd}

\newtheorem{theorem}{Theorem}[section]
\newtheorem{proposition}[theorem]{Proposition}
\newtheorem{definition}[theorem]{Definition}
\newtheorem{corollary}[theorem]{Corollary}

\newtheorem{lemma}[theorem]{Lemma}

\theoremstyle{remark}
\newtheorem{remark}[theorem]{Remark}

\numberwithin{equation}{section}
\numberwithin{figure}{section}
\numberwithin{table}{section}

\DeclareMathOperator{\Bun}{\mathrm{Bun}}
\DeclareMathOperator{\Cht}{\mathrm{Cht}}

\DeclareMathOperator{\ev}{\mathrm{ev}}
\DeclareMathOperator{\FinS}{\mathrm{FinS}}
\DeclareMathOperator{\Frob}{\mathrm{Frob}}
\DeclareMathOperator{\Gal}{\mathrm{Gal}}
\DeclareMathOperator{\GL}{\mathop{GL}}
\DeclareMathOperator{\Gr}{\mathrm{Gr}}
\DeclareMathOperator{\Hecke}{\mathrm{Hecke}}
\DeclareMathOperator{\Perv}{\mathrm{Perv}}
\DeclareMathOperator{\Reg}{\mathrm{Reg}}
\DeclareMathOperator{\Rep}{\mathrm{Rep}}

\DeclareMathOperator{\Std}{\mathrm{Std}}
\DeclareMathOperator{\Weil}{\mathrm{Weil}}

\title{Unipotent nearby cycles and the cohomology of shtukas}
\author{Andrew Salmon}
\email{asalmon@mit.edu}
\address{Department of Mathematics\\
Massachusetts Institute of Technology\\
Cambridge, MA 02139\\ USA}
\date{September 28, 2021} %
\begin{document}

\begin{abstract}
    We give cases in which nearby cycles commutes with pushforward from sheaves on the moduli stack of shtukas to a product of curves over a finite field.  The proof systematically uses the property that taking nearby cycles of Satake sheaves on the Beilinson-Drinfeld Grassmannian with parahoric reduction is a central functor together with a ``Zorro's lemma'' argument similar to that of Xue \cite{xue2020smoothness}.  As an application, for automorphic forms at the parahoric level, we characterize the image of tame inertia under the Langlands correspondence in terms of two-sided cells.
\end{abstract}

% \keywords{shtukas, nearby cycles, parahoric subgroup}

\maketitle

% \tableofcontents

\section{Introduction}

Shtukas first appeared in the work of V.~Drinfeld to tackle problems related to the Langlands correspondence for function fields, including the global Langlands correspondence for $\GL_2$ and the Ramanujan-Petersson conjecture \cite{drinfeld1988proof}.  Shtukas function realize a function field Langlands correspondence in their cohomology much in the same way that elliptic modules and elliptic sheaves do.  Shtukas have advantages over elliptic sheaves in that unlike elliptic sheaves, shtukas impose no conditions at a fixed place $\infty$ of the function field $K$ and thus can directly access, for example, cuspidal automorphic forms which are unramified everywhere.  Some of the early applications of shtukas realize the Langlands correspondence in their cohomology by applying Arthur's trace formula.  This strategy features extensively in the use of shtukas by L.~Lafforgue, for example, to prove the global Langlands correspondence for $\GL_n$ \cite{lafforgue2002chtoucas}.  The early moduli stacks of shtukas used by Drinfeld and Lafforgue were generalized by Y.~Varshavsky to allow for any reductive group $G$ and arbitrary modification type by a tuple of representations of the Langlands dual group \cite{varshavsky2004moduli}.  In Varshavsky's generalization, a shtuka is a principal $G$-bundle over a curve with modifications at `legs', which are $S$-points of the base curve, and an identification between the $G$-bundle after modification and the Frobenius twist of the $G$-bundle before modification.  V.~Lafforgue used relationships between these generalized moduli stacks, especially relationships involving creating, annihilating, and fusing different legs of the shtukas in order to give an automorphic to Galois direction of the Langlands correspondence which avoids the use of the trace formula entirely \cite{lafforgue2018chtoucas}.

V.~Lafforgue's automorphic to Galois direction of the Langlands correspondence gives, for $G$ a split reductive group over a function field $K$, an action of a commutative algebra $\mathcal{B}$ of excursion operators on the space of cuspidal automorphic forms fixed under an open compact subgroup of the adeles with finite central character.  Since the spectrum of the excursion algebra corresponds to semisimplifications of Weil group representations, this associates a Langlands parameter to cuspidal automorphic forms in the spectral decomposition of the space of automorphic forms as a module over $\mathcal{B}$.  Excursion operators commute with the action of the Hecke algebra.  If $Q$ is an open compact subgroup of the adele group $G(\mathbb{A}_K)$ corresponding to a level structure $N$, Lafforgue's work can be viewed as constructing a map
\[
\begin{tikzcd}
\left\{ \begin{gathered}\mbox{cuspidal automorphic} \\ \mbox{representations of $G$} \\
\mbox{$\pi^Q \subset C^\infty_c(\Bun_{G,N}(\mathbb{F}_q), \overline{\mathbb{Q}_{\ell}})$} \end{gathered} \right\} \arrow[r] & \left\{ \begin{gathered} \widehat{G}(\overline{\mathbb{Q}_{\ell}})\mbox{-valued representations} \\ {of} \Weil(\overline{K} / K). \end{gathered} \right\}
\end{tikzcd}
\]
A cuspidal automorphic representation contains a nonzero vector $v \in \pi^Q$ fixed under an open compact subgroup $Q \subset G(\mathbb{A}_K)$.  Given such a vector for open compact subgroup, we expect that the corresponding $\Weil(\overline{K} / K)$ representation is compatible with the level structure $Q$.  Specifically, under V.~Lafforgue's map, we would like to control the ramifications of the Galois representations in terms of the level structures of the automorphic representations.  Our main theorem is a result of this type for parahoric level structure.
\begin{theorem}
Let $\pi$ be a cuspidal automorphic representation with finite order central character for the split group $G(\mathbb{A}_K)$ where $K$ is the function field of the smooth connected curve $C$.  Suppose $\pi$ contains a nonzero vector $v \in \pi^Q$ fixed under a group
\begin{equation} Q = \prod_{x \in C} Q_x \subset {\prod_{x \in C}}' G(K_x) = G(\mathbb{A}_K) \end{equation}
where $Q_p$ is a standard parahoric subgroup for fixed $p \in C$.  Let $I^1_{K_p} \subset I_{K_p} \subset \Weil(\overline{K_p} / K_p) \subset \Weil(\overline{K} / K)$ be the wild inertia, inertia, and local Weil group at the place $p$.
\begin{enumerate}
    \item (Special case of Theorem~\ref{thm:unipotent}) Let $\rho$ be the $\Weil(\overline{K} / K)$ representation corresponding to $\pi$ under V. Lafforgue's construction.  Then $\rho(I^1_{K_p}) = 1$ and for $\gamma \in I_{K_p} / I^1_{K_p}$ a topological generator, $\rho(\gamma) \subset \widehat{G}$ is unipotent.
    \item (Theorem~\ref{thm:twosided}) Keeping the above assumptions, if $Q_p$ is a standard parahoric containing a fixed Iwahori $I$, and if $w_P$ is the longest element of the Weyl group of the Levi $Q_p / Q_p^+$ where $Q_p^+$ denotes the pro-unipotent radical of $Q_p$, and $\overline{u_P}$ is the unipotent orbit closure corresponding to the two-sided cell containing $w_P$, then $\rho(\gamma) \in \overline{u_P}$.
\end{enumerate}
\end{theorem}

In the ``spherical'' case when $Q_p$ is $G(O_p)$ for $O_p$ the ring of integers in $K_p$, it is known that $\rho(I_{K_p}) = 1$.  Recently, C.~Xue gave a proof of this result \cite{xue2020smoothness} by proving a stronger result: smoothness of cohomology sheaves of shtukas with arbitrarily many legs.  Let $C$ be a smooth connected curve over $\mathbb{F}_q$ and $\Cht_{G, N, I, W}$ be the moduli stack of shtukas for the group $G$ with $|I|$ legs with relative position given by an $I$-tuple of representations $W \in \Rep(\widehat{G}^I)$, with level structures given by a divisor $N \subset C$, living over the generic fiber of the product $C^I$.  Xue's result implies that the action of $\Weil(K, \overline{K})^I$ on $IH^*(\Cht_{G, N, I, W}, \overline{\mathbb{Q}_{\ell}})$ factors through $\Weil(C \setminus N, \overline{\eta})^I$.  An important part of Xue's technique uses the fusion structure of the perverse sheaves on $\Cht_{G, I}$ together with Zorro's lemma to construct an inverse to the natural specialization map for the sheaves
\begin{equation} \mathfrak{sp}^* \colon \mathcal{H}^j_{I, W} |_{\overline{s}} \to \mathcal{H}^j_{I, W} |_{\overline{\eta}}. \end{equation}
Such an inverse map is known automatically if one has the existence of a compactification of $\Cht_{G, N, I, W}$ such that the singularities of the compactification are no worse than the singularities of $\Cht_{G, N, I, W}$.  Constructing such a compactification seems to be a challenging problem, and a construction which satisfies the desired properties concerning singularities only seems to be available in the cases $G = \GL_2$, $I = 2$, $W = \Std \boxtimes \Std^*$ \cite{drinfeld1987cohomology} and $G = \GL_n$, $N = \varnothing$, $I = 2$, $W = \Std \boxtimes \Std^*$ \cite{lafforgue1998compactification}.

A key intermediate result is an analogue of Xue's smoothness result for the case of parahoric level structure, using unipotent nearby cycles instead of the specialization map.  At parahoric level structure, we prove that the local Weil group actions factor through tame inertia, and tame inertia acts unipotently.  Analogous to Xue's result, we use the fusion structure and Zorro's lemma in a systematic way to construct an inverse to the natural map (the notation will be explained later)
\begin{equation} \mathrm{can} \colon R\mathfrak{p}_! \Psi_J \mathscr{F}_{I \cup J, W \boxtimes V} \to \Psi_J R\mathfrak{p}_! \mathscr{F}_{I \cup J, V \boxtimes W}. \end{equation}
Once again, a standard fact about nearby cycles is that this canonical map is an isomorphism if $\mathfrak{p}$ is proper, and we use similar techniques to those of Xue in order to avoid the need to construct compactifications.  This fact on nearby cycles commuting with pushforward is suggested as a generalization of Xue's result on specialization maps above, and it also would appear as a consequence of a constellation of conjectures of Zhu \cite[Example~4.8.12]{zhu2020coherent} relating the cohomology of local and global shtukas to coherent sheaves on the stacks of local and global Langlands parameters in the function field case.

To conclude useful results about the cohomology of moduli spaces of shtukas, we need integral models, at least at the parahoric level.  Such integral models and their local models are available in large generality by combining the results of \cite{mayeux2020n} and \cite{rad2016local}.  For this reason, for the purpose of the paper, we do not assume that $G$ is a reductive group, but only that it is a smooth group scheme over a smooth projective curve $C$ over a finite field $\mathbb{F}_q$, and that it has reductive generic fiber.  The moduli stack of shtukas $\Cht_G$ will automatically incorporate the level structures, constructed in general via dilatation.  An overview of this work will occupy section \ref{sec:modulishtukas}.

In section \ref{section:excursion} we are interested in the analogue of Xue's constancy results for \'{e}tale sheaves on the product of the generic fiber of a curve.  Such a result is key for proving smoothness results and relies on the use of the Eichler-Shimura relation.  Indeed, most of the section is dedicated to writing down a general context in which the Eichler-Shimura relation holds.

In place of the fusion structure of the affine Grassmannian, we use a fusion property of nearby cycles, first identified by Gaitsgory in the proof that the nearby cycles functor is a central functor from the spherical Hecke category to the Iwahori Hecke category \cite{gaitsgory2004braiding}.  In section \ref{sec:centralsheaves}, we explain Gaitsgory's construction and subsequent generalizations to prove the desired fusion property on nearby cycles over the moduli stack of shtukas.  This is used in section \ref{sec:nearby} to prove cases where nearby cycles commutes with the pushforward $R\mathfrak{p}_!$.

Finally, in section \ref{sec:Langlands}, we draw conclusions about the Langlands correspondence and prove our main result.  The result of two-sided cells uses the work of Lusztig and Bezrukavnikov to characterize the monodromy of nearby cycles on the Beilinson-Drinfeld Grassmannian at parahoric level, the local model to draw analogous conclusions about nearby cycles on the stacks of shtukas, and the framed Langlands parameters of Lafforgue and Zhu to draw conclusions about the Langlands correspondence and the image of the tame generator.

\subsection*{Acknowledgements}

I would like to thank my advisor, Zhiwei Yun, for introducing me to shtukas and for many discussions and suggestions about this document.  I thank Cong Xue for encouraging me to pursue questions of nearby cycles at the parahoric level and for making many helpful comments.  I thank Xinwen Zhu for explaining the connection to \cite{zhu2020coherent}.  I would also like to thank the referee for pointing out a gap in the previous version of the paper in the proofs of Proposition~\ref{nearby:shtuka} and Lemma~\ref{monoidalsheaves}.

\section{Moduli spaces of parahoric shtukas}\label{sec:modulishtukas}

We introduce some standard notation for shtukas, first developing our maps in the context of a smooth affine group scheme over a curve, then specializing to N\'{e}ron blow-ups of reductive groups to show how these maps allow us to produce perverse sheaves on our stacks of shtukas \cite{mayeux2020n}.  We fix a base field $\mathbb{F}_q$ and a smooth connected projective curve $C$ over that base field.  For $G$ be a smooth affine group scheme over a curve, we can form the stack $\Bun_G$ of $G$-torsors.  We also define the Hecke stack and the Beilinson-Drinfeld Grassmannian.
\begin{definition}
Let $I$ be a fixed finite set and $G$ be a smooth affine group scheme over $C$.  For a scheme $S$ over $\mathbb{F}_q$, the ind-stack $\Hecke_{G,I}(S)$ classifies the data of
\begin{enumerate}
    \item Two bundles $\mathcal{E}_1$ and $\mathcal{E}_2$ in $\Bun_G(S)$,
    \item A collection of points $(x_i)_{i \in I} \in C^I(S)$,
    \item An isomorphism $\tau \colon \mathcal{E}_1|_{C \setminus \bigcup_i \Gamma_{x_i}} \simeq \mathcal{E}_2|_{C \setminus \bigcup_i \Gamma_{x_i}}$ away from the graphs of $x_i$.
\end{enumerate}
\end{definition}

\begin{definition}
The global (Beilinson-Drinfeld) affine Grassmannian is the fiber of the trivial bundle for the map $\Hecke_{G,I} \to \Bun_G$ sending the tuple $((x_i)_{i \in I}, \mathcal{E}_1, \mathcal{E}_2, \tau)$ to $\mathcal{E}_2$.
\end{definition}

We will also need to use certain moduli stacks of iterated affine Grassmannians.

\begin{definition}
The iterated global (Beilinson-Drinfeld) affine Grassmannian $\Gr_{I_1 \cup I_2}^{(I_1, I_2)}$ classifies tuples $((x_i)_{i \in I_1 \cup I_2}, \mathcal{E}, \mathcal{F}, \tau_1, \tau_2)$ where $\tau_1 \colon \mathcal{E} \to \mathcal{F}$ is a Hecke modification along $(x_i)_{i \in I_1}$ and $\tau_2 \colon \mathcal{F} \to G$ is a modification along $(x_i)_{i \in I_2}$, where $G$ here denotes the trivial $G$-bundle.
\end{definition}

By sending $\tau_1$ and $\tau_2$ to their composition, we get a map $c \colon \Gr_{I_1 \cup I_2}^{(I_1, I_2)} \to \Gr_{I_1 \cup I_2}$ that is proper.  Moreover, $\Gr_{I_1 \cup I_2}^{(I_1, I_2)}$ is $\Gr_{I_1} \times \Gr_{I_2}$ locally in the \'{e}tale topology.

Moreover, we can consider the moduli stack of global $G$-shtukas, which we denote $\Cht_{G, I}$.

\begin{definition}
The moduli stack of global $G$-shtukas, $\Cht_{G, I}$, is defined by the pullback square
\begin{equation}
    \begin{tikzcd}
    \Cht_{G, I} \arrow[r] \arrow[d] & \Hecke_{G, I} \arrow[d] \\
    \Bun_G \arrow[r] & \Bun_G \times \Bun_G
    \end{tikzcd}
\end{equation}
Let us explain what the arrows are.  The lower arrow is the graph of Frobenius pullback that on $S$-points sends a bundle $\mathcal{E}$ on $C \times S$ to $(\mathcal{E}, (1 \times \Frob_S)^* \mathcal{E})$ as a pair of $G$-bundles on $C \times S$.  The vertical arrow is the forgetful map that on $S$-points sends a tuple $((x_i)_{i \in I}, \mathcal{E}_1, \mathcal{E}_2, \tau)$ to $(\mathcal{E}_1, \mathcal{E}_2)$.
\end{definition}

We will adopt the convention of denoting $(1 \times \Frob_S)^*\mathcal{E}$ as ${}^{\tau} \mathcal{E}$.

Let us now be more specific about the sorts of smooth affine group schemes that we are talking about.  We consider groups $G$ over our curve $C$ which have generic fiber a quasi-split reductive group $G_K$.  Following Lafforgue \cite[Section~12]{lafforgue2018chtoucas}, there is an open $U \subset C$ such that $G_K$ extends to a smooth affine group scheme over $C$ that is a reductive group scheme over $U$ with parahoric reduction at $C \setminus U$.  Let $R$ be the complement $C \setminus U$.

Next, we can modify the group by dilatation.  Let $N$ be an effective divisor, not necessarily reduced, and let $H \subset G|_N$ be a closed, smooth group subscheme of the restriction of $G$ to the divisor $N$.  The construction of \cite{mayeux2020n} produces a smooth group scheme $G$ over $C$ by dilatation such that $\Cht_{G}$ is an integral model of shtukas for the quasi-split group with level structure according to $H$ along the divisor $N$.  Let us fix some divisor $N$ and let $\widehat{U}$ be $U \setminus N$ so that the complement of $\widehat{U}$ is $\widehat{N} := N \cup R$.  Let $P$ be the subset of $\widehat{N}$ where the dilatated group $G$ has parahoric reduction.

After dilatation, we have produced a smooth affine group scheme $G$ such that $G(\mathbb{A})$ is the adeles of our quasi-split group, which contains an open subgroup $G(\mathbb{O})$ corresponding to the level structures $H$ along $N$.  Let $Z$ be the center of $G$.  To make our moduli spaces of shtukas finite type, we need to quotient by the action of a discrete group $\Xi$, which is a fixed lattice in $Z(\mathbb{A})$ such that $\Xi \cap Z(\mathbb{O}) Z(F) = \{ 1 \}$.  This acts on $\Cht_G$ such that the quotient $\Cht_G / \Xi$ has finitely many components.  Moreover, this action can be made compatible with Harder-Narasimhan truncation, and after truncation, the stacks $\Cht_{G, I, W}^{\le \mu} / \Xi$ are of finite type.

We now want to construct sheaves over these finite type stacks.  We can consider a global version of the positive loop group that we write $G_{I, \infty}$, living over $C^I$.  There is an important map $\epsilon_G \colon \Cht_{G, I} \to [G_{I, \infty} \backslash \Gr_{G, I}]$.  This follows as a straightforward generalization of \cite[Definition~1.1.13]{xue2020cuspidal}.

\begin{definition}
For a divisor $D$, we can view the divisor as a subscheme (with non-reduced structure around points with multiplicity) and pull back the trivial $G$-bundle to $D$.  We let $G_{I, d}$ classify tuples $((x_i)_{i \in I}, g)$ where $x_i \in C(S)$ gives a tuple of points and $g$ gives an automorphism of the trivial bundle on the graph of $\sum dx_i$.

Let $G_{I, \infty} = \varprojlim_d G_{I, d}$.  We have a map
\begin{equation}
\epsilon_G \colon \Cht_{G, I} \to [G_{I, \infty} \backslash \Gr_{G, I}]
\end{equation}
defined by starting with the isomorphism
\begin{equation}
\tau \colon \mathcal{E}|_{C \setminus \bigcup_i \Gamma_{x_i}} \to {}^{\tau} \mathcal{E}|_{C \setminus \bigcup_i \Gamma_{x_i}},
\end{equation}
restricting it around formal neighborhoods of $(x_i)_{i \in I}$, and forgetting the relationship of ${}^{\tau} \mathcal{E}$ and $\mathcal{E}$.  Since every $G$-bundle restricted to formal discs is trivial, $[G_{I, \infty} \backslash \Gr_{G, I}]$ classifies tuples $((x_i)_{i \in I}, \mathcal{E}, \mathcal{F}, \tau)$ where $\mathcal{E}$ and $\mathcal{F}$ are $G_{I, \infty}$-torsors around the points $(x_i)$ and $\tau$ is an isomorphism along the punctured formal disks around $(x_i)$.
\end{definition}

Given a lattice $\Xi$ in $Z(F) \backslash Z(\mathbb{A})$ as above, $\epsilon_G$ can be made $\Xi$-equivariant and yields a map \cite[Equation~1.7]{xue2020cuspidal}
\begin{equation}
\epsilon_G^{\Xi} \colon \Cht_{G, I} / \Xi \to [G^{ad}_{I, \infty} \backslash \Gr_{G, I} ].
\end{equation}

We view $\Cht_{G, I} / \Xi$ and all of its substacks as living over $X^I$ with respect to a structure map $\mathfrak{p}$.  We note that the map $\mathfrak{p}$ often fails to be proper.

\begin{definition}
Let $\mathfrak{p} \colon \Cht_{G, I} / \Xi \to X^I$ send the tuple $((x_i)_{i \in I}, \mathcal{E}, \tau)$ to $(x_i)_{i \in I}$.
\end{definition}

Geometric Satake for ramified groups gives a functor
\begin{equation}
\Rep({}^{L}G^I) \to \Perv_{G_{I, \infty}}(\Gr_{G, I}|_{\widehat{U}^I}),
\end{equation}
and we denote the sheaves produced by $\mathcal{S}_{I, W}$.  If $I_1 \cup I_2 = I$ and $W = V_1 \boxtimes V_2$, then $\mathcal{S}_{I,W} \cong c_! \mathcal{S}_{I_1, V_1} \tilde{\boxtimes} \mathcal{S}_{I_2, V_2}$, where $c$ is the convolution map.  Here, $\mathcal{S}_{I_1, V_1} \tilde{\boxtimes} \mathcal{S}_{I_2, V_2}$ is equivalent in the \'{e}tale topology to the sheaf $\mathcal{S}_{I_1, V_1} \boxtimes \mathcal{S}_{I_2, V_2}$ on $\Gr_{G,I_1} \times \Gr_{G,I_2}$: under an \'{e}tale correspondence relating the iterated affine Grassmannian with the product, the pullbacks of both sheaves are the same.

We can pull back the equivariant sheaves $\mathcal{S}_{I,W}$ along the map $\epsilon^{\Xi}_G$ to sheaves on $\Cht_{G, I}|_{U^I}$.  The details of the pullback and compatibility with $\Xi$ are explained in more detail in \cite[Definition~2.4.7]{xue2020cuspidal}.  We denote these sheaves by $\mathscr{F}_{I, W}$ and define the closed substacks $\Cht_{G, \Xi, I, W}$ as the supports of these sheaves.  Similarly, we denote the supports of the sheaves $\mathcal{S}_{I, W}$ by $\Gr_{G, I, W}$.

We now recall the local models theorem, which is \cite[Theorem~3.2.1]{rad2016local}.  We state for the stacks $\Cht_{G, \Xi, I, W}$ and $\Gr_{G, I, W}$, as the stratification by $W$ forms a bound in the sense of \cite[Definition~3.1.3]{rad2016local}.
\begin{theorem}
For every geometric point $y$ of $\Cht_{G, I, W}$ there is an \'{e}tale neighborhood $U_y$ such that
\begin{equation}
\begin{tikzcd}
\Cht_{G, \Xi, I, W} & U_y \arrow[l] \arrow[r] & \Gr_{G, I, W}
\end{tikzcd}
\end{equation}
is a diagram with both arrows \'etale.
\end{theorem}

\begin{proof}
The statement with $\Xi$ follows by checking that the closures of strata by $W$ form a bound in the sense of \cite[Theorem~3.2.1]{rad2016local}, which is true because the sheaves produced by geometric Satake are equivariant and their supports are stable under the action of $G_{I, \infty}$.  So it suffices to check that the left and right arrow can be made $\Xi$-equivariant.  But the left arrow is a cover coming from adding level structures, and such covers are $\Xi$-equivariant, and the right arrow is a version of $\epsilon_G$.
\end{proof}

The sheaves $\mathcal{S}_{I, W}$ and $\mathscr{F}_{I, W}$ have well-known fusion properties.  In \cite{lafforgue2018chtoucas}, V. Lafforgue constructed the automorphic to Galois direction of the global Langlands correspondence by using the fusion and partial Frobenius structure in the moduli stack of shtukas.  Recently, C. Xue proved that the sheaves $R\mathfrak{p}_! \mathscr{F}_{I, W}$ form an ind-local system, and not just an ind-constructible sheaf, on $U^I$.  Our main result extends these ideas to describe how the local Galois groups act at points $p \in C \setminus U$ when $G$ has parahoric reduction at $p$.  To do this, we will use nearby cycles along the base curves in various directions $i \in I$.  Our notational convention will be to suppress the fact that $\Psi_i = R\Psi_i$ is a derived functor and also suppress similar notation for tensor products $\otimes = \otimes^L$.

\section{Excursion operators on nearby cycles}\label{section:excursion}

Let $\FinS$ be the category of finite sets with set maps as morphisms.  There are certain functors between categories cofibered over $\FinS$ that play an important role in the theory of shtukas.  On a formal level, these categories are essential for running parts of V. Lafforgue's strategy to construct excursion operators, as we will see below.

By sending $I$ to $C^I$, we get a contravariant functor $C^{\bullet}$ from sets to schemes.  For a group $\widehat{G}$ instead of $C$, we get a contravariant functor $\widehat{G}^{\bullet}$ from $\FinS$ to algebraic groups.

\begin{definition}
We consider $\Rep(\widehat{G}^{\bullet})$ into a category cofibered over $\FinS$.  The objects in this category are pairs $(I, W)$ where $I$ is a finite set and $W$ is a direct sum of $I$-tuples of representations of $\widehat{G}$ in finite dimensional $E$-vector spaces.  Morphisms in this category are generated by two types of arrows.  $I$-tuples of homomorphisms of $\widehat{G}$-representations, $W \to V$, give arrows $(I, W) \to (I, V)$.  A morphism of finite sets $\zeta \colon I \to J$ gives a fusion map $(I, W) \to (J, W^{\zeta})$ where $W^{\zeta}$ gives tensor products over the finite fibers of the map $\zeta$ where the empty tensor product yields the trivial representation $1$.

The functor that forgets the representation data and just remembers the finite set gives $\Rep(\widehat{G}^{\bullet}) \to \FinS$.  It is straightforward to show that $\Rep(\widehat{G}^{\bullet})$ is a cofibered category over $\FinS$ and the fusion maps are precisely the cocartesian arrows.
\end{definition}

\begin{definition}
Let us describe another category cofibered over $\FinS$, which we call $D^b_c(X^{\bullet})$ for a variety $X$ over $\mathbb{F}_q$.  Objects in this category are pairs $(I, \mathcal{L})$ of a finite set $I$ and an object in the bounded derived category of constructible sheaves $\mathcal{L}$ in $D^b_c(X^I)$.  Morphisms in this category are generated by morphisms in $D^b_c(X^I)$ for all $I$ together with fusion morphisms which are all of the form $(I, \mathcal{L}) \to (J, \Delta_{\zeta}^* \mathcal{L})$ where $\Delta_{\zeta} \colon X^J \to X^I$ is induced by $\zeta \colon I \to J$.  The fusion morphisms are the cocartesian arrows in the cofibered category $D^b_c(X^{\bullet})$.
\end{definition}

Geometric Satake produces a full subcategory of $\Perv_{G_{\infty, \bullet}}(\Gr_{G, \bullet}|_{C \setminus \widehat{N}})$ that is equivalent to $\Rep(\widehat{G}^{\bullet})$ as a category cofibered over $\FinS$.  The essential image consists of the sheaves $\mathcal{S}_{I, W}$ ranging over finite sets $I$ and $W \in \Rep(\widehat{G}^I)$.  Let us briefly discuss a technical point about shifts, which is that it is our convention that the sheaf $\mathcal{S}_{I, W}[-|I|]$ (and not $\mathcal{S}_{I, W}$ itself) lives in the heart of the perverse $t$-structure in order to ensure that the cocartesian property holds without shift (this is consistent with the convention in \cite{lafforgue2014introduction}, for example).  This fact is the reason why we do not add shifts in our definition of unipotent nearby cycles because we are not actually interested in giving a t-exact functor with respect to the perverse t-structure, but only one up to a shift.

Pulling back under $\epsilon^{\le \mu, \Xi}_G$ produces sheaves $\mathscr{F}_{I, W}$ on $\Cht^{\le \mu}_{G,I}|_{(X \setminus \widehat{N})^I} / \Xi$ as above, compatible under specialization to diagonals and thus a full subcategory of perverse sheaves on $\Cht^{\le \mu}_{G, \bullet}|_{(C \setminus \widehat{N})^{\bullet}} / \Xi$ cofibered over $\FinS$ generated by the sheaves $\mathscr{F}_{I, W}$.  Taking the proper pushfoward $R\mathfrak{p}_! \mathscr{F}_{I, W}$ gives a cocartesian functor
\[ \mathcal{H}^{\le \mu}_{\bullet} \colon \Rep(\widehat{G}^{\bullet}) \to D^b_c((C \setminus \widehat{N})^{\bullet}) \]
between categories cofibered over $\FinS$.  This description includes both the functoriality of $\mathcal{H}$ in the representation $W$ and also the fusion morphisms, which, by the fact that $\mathcal{H}^{\le \mu}$ is cocartesian, must be isomorphisms
\[\mathcal{H}^{\le \mu}_{J, W^{\zeta}} \cong \Delta_{\zeta}^* \mathcal{H}^{\le \mu}_{I, W}.\]

We also note that in addition to this structure described above we have commuting partial Frobenius maps $H^{\le \mu}_{I, W} \to H^{\le \mu + \kappa}_{I, W}$.  It is to such a context that we want to abstract the formation of excursion operators, the Eichler-Shimura relation \cite[Section~7]{lafforgue2018chtoucas}, and the smoothness results of \cite[Section~1]{xue2020smoothness}.  Ultimately, we will apply our results to sheaves of the form
\[ Rp_! \Psi_{j_1} \cdots \Psi_{j_n} \mathscr{F}_{I \cup J, W \boxtimes V} \]
for $J = \{ j_1, \dots, j_n \}$.

Let $\mu$ take values in a commutative monoid $S$ with order such that $0$ is minimal and it is a directed poset.  In the cases we consider, $\mu$ takes values in dominant coroots and corresponds to the Harder-Narasimhan filtration.

\begin{definition}
With the above conventions on $\mu$, an $S$-filtered collection of functors $\mathcal{C}_1 \to \mathcal{C}_2$ is a functor from the poset of $S$ to the category of functors from $\mathcal{C}_1$ to $\mathcal{C}_2$.
\end{definition}

In the cases we consider, $\mathcal{C}_2$ is $D^b_c(X^{\bullet})$, $S$ is the monoid of dominant coroots of our reductive group $G$, and we will suppress the dependence on dominant coroots.  For such filtered collections of functors $\mathscr{G}^{\le \mu}$, $\varinjlim_{\mu \in S} \mathscr{G}^{\le \mu}$ will a functor from $\mathcal{C}_1$ to $\mathrm{Ind}(\mathcal{C}_2)$.  Since the limiting cohomology sheaves as $\mu$ gets arbitrarily large are only ind-constructible sheaves, we need a tool to get control over these limits so we can treat them as if they were constructible.  Such control will be given by the Eichler-Shimura relation.

\begin{definition}
Let $\mathcal{G}^{\le \mu}$ be a filtered collection of cocartesian functors from $\Rep(\widehat{G}^{\bullet}) \to D^b_c(X^{\bullet})$ cofibered over $\FinS$, and let $v$ be a closed point of $X$.  We say that $\mathcal{G}^{\le \mu}$ is filtered with respect to partial Frobenius maps $F_{\bullet}$ if for $\kappa$ sufficiently large only depending on the isomorphism class of $\bigoplus \bigotimes_i W_i \in \Rep(\widehat{G})$ we have maps of sheaves
\begin{equation}
F_{I_0} \colon \Frob_{I_0}^{*} \mathcal{G}^{\le \mu}_{I, W} \to \mathcal{G}^{\le \mu + \kappa}_{I, W}
\end{equation}
that are natural transformations that commute with each other for $i \in I$ and such that $F_{I_0}$ together with the canonical map
\begin{equation}
\mathcal{G}^{\le \mu + \kappa}_{I, W} \to \mathcal{G}^{\le \mu + |I_0|\kappa}_{I, W}
\end{equation}
is $\prod_{i \in I_0} F_{\{ i \}}$.
In particular, they should be compatible with the fusion isomorphisms.  We denote $F_{I_0} = \prod_{i \in I_0} F_{\{ i \}}$.  If $\zeta \colon I \to J$, and $J_0 \subset J$ such that $I_0 = \zeta^{-1}(J_0)$, we want to demand that
\begin{equation}
\begin{tikzcd}
\Delta_{\zeta}^* \mathrm{Fr}_{I_0}^* \mathcal{G}^{\le \mu}_{I, W} \arrow[r, "\Delta_{\zeta}^* F_{I_0}"] \arrow[d, "\mathrm{Fr}_{J_0}^* \mathrm{fus}_{\zeta}"]& \Delta_{\zeta}^* \mathcal{G}^{\le \mu + \kappa}_{I, W} \arrow[d, "\mathrm{fus}_{\zeta}"] \\
\mathrm{Fr}_{J_0}^* \mathcal{G}^{\le \mu}_{J, W^{\zeta}} \arrow[r, "F_{J_0}"]& \mathcal{G}^{\le \mu + \kappa}_{J, W^{\zeta}}.
\end{tikzcd}
\end{equation}

Finally, we demand that in the case $|I|=1$, that $F_I$ factors as the composition of the action of Frobenius
\begin{equation} \Frob^{*} \mathcal{G}^{\le \mu}_{I, W} \to \mathcal{G}^{\le \mu}_{I, W} \end{equation}
and the canonical map
\begin{equation} \mathcal{G}^{\le \mu}_{I, W} \to \mathcal{G}^{\le \mu + \kappa}_{I, W}. \end{equation}
\end{definition}

Now suppose the base $X$ is a dense open subset of a smooth projective curve $C$.  We note that passing to the limit $\varinjlim \mathcal{G}^{\le \mu}$ and taking the stalk at a geometric generic point, our assumptions show that the maps $F_{i}^n$ give rise to an action of
\[ F\Weil(\eta_I, \overline{\eta_I}) \]
which is an extension of the kernel of $\Gal(\overline{F_I} / F_I) \to \widehat{\mathbb{Z}}$ by $\mathbb{Z}^I$ \cite[Remarque~8.18]{lafforgue2018chtoucas}.  By Drinfeld's lemma, modules finitely generated under the action of some algebra commuting with partial Frobenius factor through $\Weil(\eta, \overline{\eta})^I$.  We would like to show that this factorization occurs over the whole $\varinjlim \mathcal{G}^{\le \mu}$.  This is the main result of the section, and the proof consists of applying the proof of \cite[Proposition~1.4.3]{xue2020smoothness}.  We note that everything in the proof goes through as long as we have an analogue of the Eichler-Shimura relation.  The majority of the remaining section will be devoted to establishing this result.

For this to be a useful definition, we will need an example of such a collection other than the proper pushforward $\mathcal{H}^{\le \mu}_{I, W}$ used by Lafforgue.

\begin{proposition}\label{nearby:shtuka}
As always, let $G$ be a smooth group scheme whose generic fiber is reductive, and let $p \in C$ be a closed point of parahoric reduction.  Consider the sheaves $\mathscr{F}^{\le \mu}_{I \cup K, W \boxtimes V}$ on $\Cht^{\le \mu}_{G, \Xi, I \cup K, W \boxtimes V} |_{(C \setminus \widehat{N})^I}$ as constructed above.  Let $K = \{ k_1, \dots, k_n \}$ and consider iterated nearby cycles $\Psi_{k_1} \cdots \Psi_{k_n} \mathscr{F}^{\le \mu}_{I \cup K, W \boxtimes V}$, taking the nearby cycle $\Psi_{k_i}$ with respect to the coordinate $k_i \in K$ along $(C \setminus \widehat{N})^{I \cup K}$ with special point a geometric point above $p$.

Consider the functors $\Rep(\widehat{G}^{\bullet}) \to D^b_c((C \setminus \widehat{N})^{\bullet})$ defined by
\[ R\mathfrak{p}_! \Psi_{k_1} \cdots \Psi_{k_n} \mathscr{F}^{\le \mu}_{I \cup K, W \boxtimes V}. \]
These functors are cocartesian functors between categories cofibered over $\FinS$ and are filtered with respect to partial Frobenius maps $F_{\{i\}}$ for $i \in I$.
\end{proposition}

\begin{proof}
The functoriality with respect to maps $(I, W) \to (I, V)$ for $W \to V$ a morphism in the categories $\Rep(\widehat{G}^I)$ living above the identity $1_I$ follows immediately from the functoriality of $\mathscr{F}^{\le \mu}$, nearby cycles, and proper pushforward.  So to prove that our functor exists and is cocartesian, it suffices to check the fusion property, namely that for $\zeta \colon I \to J$, the cocartesian morphisms $(I, W) \to (J, W^{\zeta})$ are sent to cocartesian morphisms
\[ \Delta_{\zeta}^* R\mathfrak{p}_! \Psi_{k_1} \cdots \Psi_{k_n} \mathscr{F}^{\le \mu}_{I \cup K, W \boxtimes V} \cong R\mathfrak{p}_! \Psi_{k_1} \cdots \Psi_{k_n} \mathscr{F}^{\le \mu}_{J \cup K, W^{\zeta} \boxtimes V}. \]
Note that we always have a canonical map going in one direction
\[ \Delta_{\zeta}^* R\mathfrak{p}_! \Psi_{k_1} \cdots \Psi_{k_n} \mathscr{F}^{\le \mu}_{I \cup K, W \boxtimes V} \to R\mathfrak{p}_! \Psi_{k_1} \cdots \Psi_{k_n} \Delta_{\zeta}^* \mathscr{F}^{\le \mu}_{I \cup K, W \boxtimes V} \cong R\mathfrak{p}_! \Psi_{k_1} \cdots \Psi_{k_n} \mathscr{F}^{\le \mu}_{J \cup K, W^{\zeta} \boxtimes V}. \]
by proper base change.  By the local models theorem, it suffices to construct the corresponding isomorphisms on the affine Grassmannian,
\begin{equation} \Delta_{\zeta}^* \Psi_{k_1} \cdots \Psi_{k_n} \mathcal{S}_{I \cup K, W \boxtimes V} \cong \Psi_{k_1} \cdots \Psi_{k_n} \mathcal{S}_{J \cup K, W^{\zeta} \boxtimes V}. \end{equation}

By writing $\mathcal{S}_{I \cup K, W \boxtimes V} \cong c_! \mathcal{S}_{I, W} \tilde{\boxtimes} \mathcal{S}_{K, V}$, where $c$ is the convolution map, and applying proper base change and the \'{e}tale local model property of the iterated affine Grassmannian, we can reduce to the external product case, where we need to show that
\[ \Delta_{\zeta}^* \mathcal{S}_{I, W} \boxtimes \Psi_{k_1} \cdots \Psi_{k_n} \mathcal{S}_{K, V} \cong \mathcal{S}_{J, W^{\zeta}} \boxtimes \Psi_{k_1} \cdots \Psi_{k_n} \mathcal{S}_{K, V} \]
as sheaves on $\Gr_{G, J} \times \mathrm{Fl}_p$.  This follows by the fusion property on the Beilinson-Drinfeld Grassmannian.

Now we want to show that 
\[ R\mathfrak{p}_! \Psi_{k_1} \cdots \Psi_{k_n} \mathscr{F}^{\le \mu}_{I \cup K, W \boxtimes V} \]
are filtered by partial Frobenius.  This results from the fact that partial Frobenius maps exist in our setting, and they are local homeomorphisms.  Therefore they commute with nearby cycles.  So we get an isomorphism
\begin{equation} \left( \mathrm{Fr}_{I_1}^{(I_1, \dots, I_n, K)} \right)^* \left( \Psi_{k_1} \cdots \Psi_{k_n} \mathscr{F}^{\le \mu + \kappa}_{I \cup K, J \boxtimes W} \right) \cong \Psi_{k_1} \cdots \Psi_{k_n} \mathscr{F}^{\le \mu}_{I \cup K, J \boxtimes W}. \end{equation}

Now, using the above isomorphism, the same cohomological correspondence as in \cite[Section~4.3]{lafforgue2018chtoucas} shows that the sheaves in our setting are filtered by partial Frobenius.
\end{proof}

The application of the above proposition will be to show that these sheaves satisfy an analogue of the Eichler-Shimura relations, which in turn will be used to prove a smoothness property their cohomology sheaves are smooth via a Drinfeld lemma argument.

\begin{definition}[Generalized $S$-excursion operators]
Let $\mathcal{G}^{\le \mu}$ be a collection of cocartesian functors from $\Rep(\widehat{G}^{\bullet}) \to D^b_c(X^{\bullet})$ cofibered over $\FinS$ and filtered with respect to partial Frobenius $F_{\bullet}$.  Let $v$ be a closed point of $X$.  We define the $S$-excursion operator as the composition of creation, $F$ on a created leg, and annihilation.
\begin{align*}
\mathcal{G}^{\le \mu}_{I, W} &\cong \mathcal{G}^{\le \mu}_{I \cup \{ 0 \}, W \boxtimes 1}|_{X^I \times v} \\
&\to \mathcal{G}^{\le \mu}_{I \cup \{ 0 \}, W \boxtimes (V \otimes V^*)}|_{X^I \times v} \\
&\cong \mathcal{G}^{\le \mu}_{I \cup \{ 1,2 \}, W \boxtimes V \boxtimes V^*}|_{X^I \times v \times v} \\
&\to_{F_1} \mathcal{G}^{\le \mu+\kappa}_{I \cup \{ 1,2 \}, W \boxtimes V \boxtimes V^*}|_{X^I \times v \times v} \\
&\cong \mathcal{G}^{\le \mu+\kappa}_{I \cup \{ 0 \}, W \boxtimes (V \otimes V^*)}|_{X^I \times v} \\
&\to \mathcal{G}^{\le \mu+\kappa}_{I \cup \{ 0 \}, W \boxtimes 1}|_{X^I \times v} \\
&\cong \mathcal{G}^{\le \mu+\kappa}_{I, W}
\end{align*}
where $\to_{F_1}$ denotes the action of the operator $F_{1}$ along the leg $1$.  We thus write $S_{V, v}$ for this operator, suppressing the dependence on $\mu$, $I$, and $W$.
\end{definition}

\begin{theorem}[Eichler-Shimura]
Let $\mathcal{G}^{\le \mu}$ be a collection of cocartesian functors from $\Rep(\widehat{G}^{\bullet}) \to D^b_c(X^{\bullet})$ cofibered over $\FinS$ and filtered with respect to partial Frobenius maps $F_{\bullet}$.
Then,
\begin{equation}\label{eq:EichlerShimura}
\sum_{i=0}^{\dim V} (-1)^i F_0^{\deg(v) i} \circ S_{\wedge^{\dim V - i} V, v} = 0
\end{equation}
as a map from $\mathcal{G}^{\le \mu}_{I \cup \{ 0 \}, W \boxtimes V}|_{X^I \times v} \to \mathcal{G}^{\le \mu + \kappa}_{I \cup \{ 0 \}, W \boxtimes V}|_{X^I \times v}$ for sufficiently large $\kappa$ depending on $V$ but not $\mu$.
\end{theorem}

This proof is essentially contained in \cite[Section~7]{lafforgue2018chtoucas} and the proof we give here involves essentially no new ideas.  For the reader's convenience, we review the argument almost verbatim from \cite{lafforgue2018chtoucas}.  As a warm-up, we review Lafforgue's point of view on the Cayley-Hamilton theorem.

Before proceeding to the proof, we first make some definitions.  In what follows, let $J$ be a finite set and $V$ a finite-dimensional vector space over $E$.  For a finite set $I$, let $V^{\otimes I}$ be the tensor power $V \otimes \dots \otimes V$.  If $T \in \mathrm{End}(V)$ and $U \in \mathrm{End}(V^{\otimes \{ 0 \} \cup J})$.

We denote by $\delta_V$ the creation map $1 \to V \otimes V^*$ and $\ev_V$ the annihilation $V \otimes V^* \to 1$.  We define $\mathfrak{C}_J(T, U)$ as the composition
\begin{equation}\label{eq:CJconstruction}
\begin{tikzcd}[column sep=huge]
V \arrow[r, "1 \otimes \delta_V^{\otimes J}"]& V \otimes (V \otimes V^*)^{\otimes J} \arrow[dl] \\
V^{\otimes \{ 0 \} \cup J} \otimes (V^*)^{\otimes J} \arrow[r, "U \otimes 1"]& V^{\otimes \{ 0 \} \cup J} \otimes (V^*)^{\otimes J} \arrow[dl] \\
V \otimes V^{\otimes J} \otimes (V^*)^{\otimes J} \arrow[r, "1 \otimes T^{\otimes J} \otimes 1"]& V \otimes V^{\otimes J} \otimes (V^*)^{\otimes J} \arrow[dl] \\
V \otimes (V \otimes V^*)^{\otimes J} \arrow[r, "1 \otimes \ev_V^{\otimes J}"]& V.
\end{tikzcd}
\end{equation}
At the end we are left with $\mathfrak{C}_J(T, U)$ as an endomorphism of $V$.  $\mathfrak{C}_J$ is linear in the argument $U$.

We apply this in the specific case that the endomorphism $U$ is the antisymmetrizer.  Recall that if $S_I$ denotes the group of permutations of $I$, and $s(\sigma)$ is the sign of a permutation $\sigma \in S_I$, then the antisymmetrizer $\mathcal{A}_I$ is
\[ \mathcal{A}_I = \frac{1}{|I|!} \sum_{\sigma \in S_{I}} s(\sigma) \sigma. \]
The antisymmetrizer naturally lives in the group algebra $\mathbb{Q}[S_I]$ which acts naturally on $V^{\otimes I}$ for $V$ a vector space.  We also refer to the antisymmetrizer $\mathcal{A}_I$ as the image of this element under the map to $\mathrm{End}(V^{\otimes I})$.

Lafforgue shows that
\[ \mathfrak{C}_J(T, \mathcal{A}_{\{ 0 \} \cup J}) = \frac{1}{n+1} \sum_{i=0}^{|J|} (-1)^i \mathrm{Tr}(\wedge^{n-i} V) T^i \]
which can be thought of as a generalization of the Cayley-Hamilton theorem, as the antisymmetrizer is the endomorphism $0$ if $|J| = \dim V$.

The Cayley-Hamilton theorem is a special case of the Eichler-Shimura relation for the functor $\Rep(\mathbb{N}^{\bullet}) \to D^b_c(X^{\bullet})$ defined by the rule that it sends a tuple of vector spaces $(\{1,\dots,n\}, (W_1, \dots, W_n))$ to the constant sheaf over $X^n$ which is an external tensor product $W_1 \boxtimes \cdots \boxtimes W_n$.  Here, partial Frobenius acts trivially as the sheaves are constant, and linear maps $T$ can be represented by the endomorphism $1 \in \mathbb{N}$ in each coordinate.  Under these identifications, the excursion operator becomes related to the trace.  Moreover, the method of proof of the Cayley-Hamilton theorem above will generalize to the general case.  The key subtleties are that we use the fusion isomorphisms to connect each of the lines in Equation \ref{eq:CJconstruction}.

In what follows that the fact that $\mathcal{G}^{\le \mu}$ is cofibered over $\FinS$ implies that we have fusion isomorphisms
\begin{equation}
    \mathrm{fus}_{\zeta} \colon \Delta_{\zeta}^* \mathcal{G}^{\le \mu}_{I, W} \to \mathcal{G}^{\le \mu}_{J, W^{\zeta}}
\end{equation}
with a natural compatibility with $F_{\bullet}$.

The functoriality of $\mathcal{G}^{\le \mu}$ induces creation maps $\mathscr{C}^{\sharp}_V$ and annihilation maps $\mathscr{C}^{\flat}_V$ as
\begin{equation}
    \mathscr{C}^{\sharp}_V \colon \mathcal{G}^{\le \mu}_{I \cup \{ 0 \}, W \boxtimes W'} \to \mathcal{G}^{\le \mu}_{I \cup \{ 0 \}, W \boxtimes (W' \times V \times V^*)}
\end{equation}
\begin{equation}
    \mathscr{C}^{\flat}_V \colon \mathcal{G}^{\le \mu}_{I \cup \{ 0 \}, W \boxtimes (W' \times V \times V^*)} \to \mathcal{G}^{\le \mu}_{I \cup \{ 0 \}, W \boxtimes W'}
\end{equation}
of sheaves over $X^{I \cup \{ 0 \}}$.  For a morphism $u \colon V \to W$ in $\Rep(\widehat{G}^I)$, we denote
\begin{equation}
\mathcal{G}^{\le \mu}(u) \colon \mathcal{G}^{\le \mu}_{I, V} \to \mathcal{G}^{\le \mu}_{I, W}
\end{equation}
the morphism induced by functoriality.  We have now defined the basic ingredients we need for our proof.

\begin{proof}
Let us define the generalization of $\mathfrak{C}_J$ to $G^{\le \mu}$.  We fix $I$ be a finite set and $W \in \Rep(\widehat{G}^I)$.  We also fix an elements $0,1,2 \not \in I \cup J$ and $V \in \Rep(\widehat{G})$.  For a map $U \colon V^{\otimes \{ 0 \} \cup J} \to V^{\otimes \{ 0 \} \cup J}$ of $\widehat{G}$-representations, let us write $\mathfrak{C}_J(G^{\le \mu}, U)$ for the composition of the four maps below.  The maps below are maps of sheaves on $X^I \times \Delta(v)$ where $\Delta(v)$ is the diagonal inclusion of the closed point $v$ in the remaining copies of $X$.
\begin{equation}\label{eq:CJdefinition}
\begin{tikzcd}[column sep=huge]
\mathcal{G}^{\le \mu}_{I \cup \{ 0 \}, W \boxtimes V} \arrow[r, "\mathscr{C}^{\sharp}_{V^{\otimes J}}"]& \mathcal{G}^{\le \mu}_{I \cup \{ 0,1,2 \}, W \boxtimes V \boxtimes V^{\otimes J} \boxtimes (V^*)^{\otimes J}} \\
\mathcal{G}^{\le \mu}_{I \cup \{ 0 \} \cup \{ 2 \}, W \boxtimes V^{\otimes \{ 0 \} \cup J} \boxtimes (V^*)^{\otimes J}} \arrow[r, "\mathcal{G}^{\le \mu}(1 \boxtimes U \boxtimes 1)"]& \mathcal{G}^{\le \mu}_{I \cup \{ 0 \} \cup \{ 2 \}, W \boxtimes V^{\otimes \{ 0 \} \cup J} \boxtimes (V^*)^{\otimes J}} \\
\mathcal{G}^{\le \mu}_{I \cup \{ 0 \} \cup J \cup \{ 2 \}, W \boxtimes V^{\boxtimes \{ 0 \} \cup J} \boxtimes (V^*)^{\otimes J}} \arrow[r, "\prod_{j \in J} F_{\{ j \}}^{\deg(v)}"]& \mathcal{G}^{\le \mu + \kappa}_{I \cup \{ 0 \} \cup J \cup \{ 2 \}, W \boxtimes V^{\otimes \{ 0 \} \cup J} \boxtimes (V^*)^{\otimes J}} \\
\mathcal{G}^{\le \mu + \kappa}_{I \cup \{ 0,1,2 \}, W \boxtimes V \boxtimes V^{\otimes J} \boxtimes (V^*)^{\otimes J}} \arrow[r, "\mathscr{C}^{\flat}_{V^{\otimes J}}"]& \mathcal{G}^{\le \mu + \kappa}_{I \cup \{ 0 \}, W \boxtimes V}.
\end{tikzcd}
\end{equation}
All the lines are connected by fusion isomorphisms, and $\kappa$ is chosen sufficiently large so that the map $\prod_{j \in J} F_{\{ j \}}$ is well defined.

Now consider the identity $\mathfrak{C}_J(\mathcal{G}^{\le \mu}, (|J| + 1) \mathcal{A}_J) = 0$ for $J = \{ 1, \dots, \dim V \}$.  We will express the left hand side as the left hand side in the Eichler-Shimura relation \ref{eq:EichlerShimura}.  For $\sigma \in S_{\{ 0 \} \cup J}$, we denote $\ell(\sigma, 0)$ the length of the longest cycle containing $0$.

We rewrite our equation as
\[ \sum_{i=0}^{\dim V} \mathfrak{C}_J\left(\mathcal{G}^{\le \mu}, \frac{1}{|J|!} \sum_{\ell(\sigma, 0) = i + 1} s(\sigma) \sigma\right) = 0. \]
We note that by symmetry and counting possible permutations,
\[ \mathfrak{C}_J\left(\mathcal{G}^{\le \mu}, \frac{1}{\dim V!} \sum_{\ell(\sigma, 0) = i + 1} s(\sigma) \sigma\right) = \mathfrak{C}_J\left(\mathcal{G}^{\le \mu}, \frac{1}{(\dim V-i)!} \sum_{\sigma \in S^{[i]}_{\{ 0 \} \cup J}} s(\sigma) \sigma\right) \]
where $S^{[i]}_{\{ 0 \} \cup J}$ denotes the set of permutations such that $\sigma(j) = j + 1$ for $0 \le j < i$ and $\sigma(i) = 0$.  The terms $\sigma$ in the sum now all factor as $\sigma = (0 \dots i) \tau$ and the $i$th term is now
\[ \mathfrak{C}_J\left(\mathcal{G}^{\le \mu}, (-1)^i (0 \dots i) \frac{1}{(\dim V-i)!} \sum_{\tau \in S_{\{ i+1, \dots, n \}}} s(\tau) \tau\right) = (-1)^i \mathfrak{C}_J\left(\mathcal{G}^{\le \mu}, (0 \dots i) \mathcal{A}_{\{ i+1, \dots, n \}}\right). \]
$\mathcal{A}_{\{ i+1, \dots, n \}}$ acts by an idempotent on $V^{\otimes \{ i+1, \dots, n \}}$ projecting onto the subspace $\wedge^{\dim V - i} V$.  The cycle $(0 \dots i)$ and $\mathcal{A}_{\{ i + 1, \dots, n \}}$ are disjoint and so the operator $\mathfrak{C}_J(\mathcal{G}^{\le \mu}, (0 \dots i) \mathcal{A}_{\{ i+1, \dots, n \}})$ splits as the composition of $\mathfrak{C}_{\{1, \dots, i\}}(\mathcal{G}^{\le \mu}, (0 \dots i))$ and $S_{\wedge^{\dim V - i} V, v}$.  So it suffices to show that
\[ \mathfrak{C}_{\{1, \dots, i\}}(\mathcal{G}^{\le \mu}, (0 \dots i)) = F_{\{ 0 \}, v}^{\deg(v) i}. \]
This will follow by induction on $i$, where the base case $i = 0$ is trivial.  We can write the left hand operation as creation in pairs $(1, 2) \cdots (2n-1, 2n)$, application of partial Frobenius to legs $0, 2, \dots, 2n-2$, then annihilation in pairs $(0, 1) \cdots (2n-2, 2n-1)$.  The creation of the pair $(2n-1, 2n)$ commutes with all the partial Frobenius, but by Zorro's lemma, the composition of that creation with annihilation at $(2n-2, 2n-1)$ is the identity.  We are left with an extra application of partial Frobenius to the leg $2n-2$ which identifies with the remaining leg $2n$ after applying Zorro's lemma.  This shows the induction step.
\end{proof}

\begin{theorem}
Let $\mathcal{G}^{\le \mu}$ be a collection of functors from $\Rep(\widehat{G}^{\bullet}) \to D^b_c(X^{\bullet})$ cofibered over $\FinS$ and filtered with respect to partial Frobenius maps $F_{\bullet}$.  Suppose also that $X$ is a Zariski dense open subset of $C_{\overline{\mathbb{F}_q}}$, where $C$ is a proper curve over a finite field.  Let $\overline{\eta}$ be a geometric generic point of $X$.
Then $\varinjlim_{\mu} \mathcal{G}^{\le \mu}_{I, W}|_{\overline{\eta}^I}$ is constant as an ind-constructible sheaf.
\end{theorem}

\begin{proof}
Now that we have the Eichler-Shimura relation, the proof of \cite[Proposition~1.4.3]{xue2020smoothness} goes through verbatim.  Since specialization maps are isomorphisms, we are done.
\end{proof}

\section{Unipotent nearby cycles over a power of a curve}\label{sec:centralsheaves}

We begin with some general theory, following Beilinson and Gaitsgory \cite{beilinson1987glue} \cite{gaitsgory2004braiding}.  Let us consider a general setup.  Let $X$ be the spectrum of a Henselian trait or a smooth curve, $I = \{ 1, \dots, n \}$ and $f \colon Y \to X^I$ a stack of finite type, and let $s \in X$ be a special point.  Let $\mathscr{G}$ be a sheaf on $Y$.  We can consider iterated unipotent nearby cycles $\Psi^u_1 \circ \dots \circ \Psi^u_n \mathscr{G}$.  We can also restrict $\mathscr{G}$ to the preimage $f^{-1}(\Delta(X)) =: Y_\Delta$ of the diagonal $\Delta(X) \subset X^I$ and perform unipotent nearby cycles $\Psi^u_{\Delta} (\mathscr{G}|_{\Delta(X)})$ there, which we will abuse notation and denote simply $\Psi^u_{\Delta} \mathscr{G}$ when the base $X^I$ is clear.  That is, $\Psi^u_{\Delta} = \Psi^u i_{Y_\Delta}^*$ where $i_{Y_\Delta} : Y_\Delta \to Y$ is the inclusion.  We can compare these functors with a functor $\Upsilon$ that generalizes Beilinson's construction of unipotent nearby cycles.

Consider first the case when our base is $1$-dimensional, that is $n = 1$.  Beilinson's construction of nearby cycles starts with local systems $\mathscr{L}_m$ on $X \setminus \{ s \}$ such that monodromy acts according to an $m \times m$ nilpotent Jordan block.  We have natural inclusions of local systems $\mathscr{L}_m \to \mathscr{L}_{m'}$ when $m \le m'$.  Let $i$ be the inclusion of $f^{-1}(s)$ and $j$ the inclusion of the complement $f^{-1}(X \setminus \{ s \})$.  One of Beilinson's results is that unipotent nearby cycles can be written as
\begin{equation}
\Psi^u \mathscr{F} \cong \varinjlim_m i^* R j_* \left( \mathscr{F} \otimes f^* \mathscr{L}_m \right).
\end{equation}

To compare the iterated unipotent nearby cycles with the diagonal unipotent nearby cycles, we introduce a functor $\Upsilon(\mathscr{G})$ that maps naturally to both.  In some favorable cases we consider, $\Upsilon(\mathscr{G})$ will be isomorphic to both $\Psi^u_{\Delta}(\mathscr{G})$ and $\Psi^u_1 \cdots \Psi^u_n$.  Let $i$ be the inclusion of $f^{-1}(\Delta(s)) = f^{-1}((s, \dots, s))$ and $j$ the inclusion of the preimage $f^{-1}((X \setminus \{ s \})^I)$.  We define our functor on objects as
\begin{equation}
    \Upsilon(\mathscr{G}) \cong \varinjlim_{m_1, \dots, m_n} i^* R j_* \left( \mathscr{G} \otimes f^* ( \mathscr{L}_{m_1} \boxtimes \cdots \boxtimes \mathscr{L}_{m_n} ) \right).
\end{equation}
These functors are written with shifts in Gaitsgory's proof that nearby cycles gives a central functor \cite{gaitsgory2004braiding}.  With shifts, the functors will send perverse sheaves to perverse sheaves.  For our purposes, it is easier to ignore the shifts, since the Satake sheaves are only perverse up to a shift (otherwise, the cocartesian property would only be true up to a shift).  The following proposition contains all the properties that we will use about the functor $\Upsilon$.
\begin{proposition}\label{prop:omnibusupsilon}
We abbreviate all mentions of unipotent nearby cycles $\Psi^u$ as $\Psi$.
\begin{enumerate}
    \item There are natural maps
    \begin{equation}
    \begin{tikzcd}
        \Psi_1 \cdots \Psi_n \mathscr{G} & \Upsilon \mathscr{G} \arrow[r] \arrow[l] & \Psi_{\Delta} \mathscr{G}|_{f^{-1}(\Delta(X))}.
    \end{tikzcd}
    \end{equation}
    between functors from $D^b(f^{-1}((X \setminus \{ s \})^I),\overline{\mathbb{Q}_\ell}) \to D^b(Y_s,\overline{\mathbb{Q}_\ell})$.
    \item Let $\pi \colon Y \to Y'$ be of finite type over $X^I$.  There is a natural transformation of functors $R \pi_! \Upsilon \to \Upsilon R \pi_!$ such that the following diagram commutes
    \begin{equation}
        \begin{tikzcd}
        R \pi_! \Psi_1 \cdots \Psi_n \arrow[d] & R \pi_! \Upsilon \arrow[l] \arrow[r] \arrow[d] & R \pi_! \Psi_{\Delta} \arrow[d] \\
        \Psi_n \cdots \Psi_n R \pi_! & \Upsilon R \pi_! \arrow[l] \arrow[r] & \Psi_{\Delta} R \pi_!.
        \end{tikzcd}
    \end{equation}
    Moreover, the vertical arrows are isomorphisms if $\pi$ is proper.
    \item Let $g \colon Y \to Y'$ be of finite type over $X^I$.  There is a natural transformation of functors $g^* \Upsilon \to \Upsilon g^*$ such that the following diagram commutes
    \begin{equation}
        \begin{tikzcd}
        g^* \Psi_1 \cdots \Psi_n \arrow[d] & g^* \Upsilon \arrow[l] \arrow[r] \arrow[d] & g^* \Psi_{\Delta} \arrow[d] \\
        \Psi_1 \cdots \Psi_n g^* & \Upsilon g^* \arrow[l] \arrow[r] & \Psi_{\Delta} g^*.
        \end{tikzcd}
    \end{equation}
    Moreover, the vertical arrows are isomorphisms if $g$ is smooth.
    \item (K\"unneth formula) Suppose that $n = 2$ and $Y \cong Y_1 \times Y_2$ and the map to $X \times X$ comes from maps $Y_i \to X$ for $i = 1,2$, and suppose that $\mathscr{G} \cong \mathscr{G}_1 \boxtimes \mathscr{G}_2$.  Then the maps $\Upsilon \to \Psi \circ \Psi$ and $\Upsilon \to \Psi_{\Delta}$ are isomorphisms, and the composition coincides with the isomorphism
    \begin{equation} \Psi_1 \mathscr{G}_1 \otimes \Psi_2 \mathscr{G}_2 \cong \Psi_{\Delta} ( \mathscr{G}_1 \boxtimes_X \mathscr{G}_2 ) \end{equation}
    observed by Gabber (see Illusie \cite[4.3]{illusie1994expose}).
\end{enumerate}
\end{proposition}

\begin{proof}
The construction of the maps is a straightforward generalization of \cite{gaitsgory2004braiding}.  We consider $\Delta_X \colon X \to X^I$ the diagonal inclusion and $j_{\Delta} \colon f^{-1}(X \setminus \{ s \}) \to Y_\Delta$ where we imagine $X \setminus \{ s \}$ as living in the diagonal of $X^I$.  Let $f_{\Delta}$ be the base change of the structure map $f$ along the diagonal $\Delta(X) \subset X^I$.  The canonical base change map
\[ \Delta_X^* R j_* \to R j_{\Delta(X),*} \Delta_{X \setminus \{ s \}}^* \]
gives a map from
\[ \Delta_X^* R j_* \left( \mathscr{G} \otimes f^* ( \mathscr{L}_{m_1} \boxtimes \cdots \boxtimes \mathscr{L}_{m_n} ) \right) \]
to
\[ R j_{\Delta(X),*} \left( \mathscr{G}|_{\Delta(X \setminus \{ s \})} \otimes f_{\Delta}^* ( \mathscr{L}_{m_1} \otimes \cdots \otimes \mathscr{L}_{m_n} ) \right). \]
Note that we have a map $\mathscr{L}_n \otimes \mathscr{L}_m \to \mathscr{L}_{\max(n,m)}$, which gives a map
\[ R j_{\Delta(X),*} \left( \mathscr{G}|_{\Delta(X \setminus \{ s \})} \otimes f_{\Delta}^* ( \mathscr{L}_{m_1} \otimes \cdots \otimes \mathscr{L}_{m_n} ) \right) \to R j_{\Delta(X),*} \left( \mathscr{G}|_{\Delta(X)} \otimes f^* ( \mathscr{L}_{\max(m_1, \dots, m_n)} ) \right). \]
After composing the two maps, pulling back along the inclusion of $i_{s,\Delta} \colon \{ s \} \to \Delta(X)$, and passing to the limit over $m_1, \dots, m_n$, the left hand side becomes $\Upsilon(\mathscr{G})$ while the right hand side becomes $\Psi_{\Delta}(\mathscr{G})$.

Let $i_k$ be the inclusion $X^{k-1} \times s^{n-k+1} \to X^k \times s^{n-k}$ along the $k$th coordinate, and let $j_k$ be the inclusion $(X \setminus \{ s \})^k \to X^k$ which we distinguish from $j^n \colon (X \setminus \{ s \})^n \to (X \setminus \{ s \})^{n-1} \times X$.  We note first that by unwinding the definition of $\boxtimes$, we have
\[ f^* ( \mathscr{L}_{m_1} \boxtimes \cdots \boxtimes \mathscr{L}_{m_n} ) \cong f^* ( \mathscr{L}_{m_1} \boxtimes \cdots \boxtimes \mathscr{L}_{m_{n-1}} \boxtimes \mathscr{L}_{1}) \otimes f^*(\mathscr{L}_{1} \cdots \boxtimes \mathscr{L}_1 \boxtimes \mathscr{L}_{m_{n}}). \]

Next, we consider the following diagram
\begin{equation}
    \begin{tikzcd}
    (X \setminus \{ s \})^{n-1} \times \{ s \} \arrow[r, "i_n"] \arrow[d, "j_{n-1} \times \{ s \}"] & (X \setminus \{ s \})^{n-1} \times X \arrow[d, "j_{n-1} \times X"] \\
    X^{n-1} \times \{ s \} \arrow[r, "i_n"] & X^n
    \end{tikzcd}
\end{equation}
to produce a map
\begin{equation}
    R(j_{n-1} \times X)_* \to R(j_{n-1} \times X)_* (i_n)_* (i^n)^* \cong (i_n)_* R(j_{n-1} \times \{ x \})_* i_n^*
\end{equation}
which by adjunction gives us $i_n^* R(j_{n-1} \times X)_* \to R(j_{n-1} \times \{ x \}) i_n^*$.  Using these maps, and writing $j$ as the composition of $(j_{n-1} \times X) \circ j^n$ we can now produce a map from
\[ i_n^* R j_{*} \left( \mathscr{G} \otimes f^* ( \mathscr{L}_{m_1} \boxtimes \cdots \boxtimes \mathscr{L}_{m_{n}} ) \right) \]
to
\[ R(j_{n-1} \times \{ s \})_* \left( i_n^* Rj^n_{*} \left( \mathscr{G} \otimes f^* ( \mathscr{L}_1 \boxtimes \cdots \boxtimes \mathscr{L}_1 \boxtimes \mathscr{L}_{m_n} ) \right) \otimes f^* \left( \mathscr{L}_{m_1} \boxtimes \cdots \boxtimes \mathscr{L}_{m_{n-1}} \boxtimes \mathscr{L}_1 \right) \right). \]
By pulling back and passing to the limit, this gives a map
\[ \Upsilon_n(\mathscr{G}) \to \Upsilon_{n-1} \Psi_n(\mathscr{G}) \]
which inductively gives the map we want.

The second and third parts of the proposition follow by proper and smooth base change.  The base change theorems can be used to give maps $R\pi_! Rj_* \to Rj_* R\pi_!$ by adjunction, which is an isomorphism if $\pi$ is a proper map.  Since the other maps in the diagram arise naturally from base change theorems, the diagrams will commute.

Finally, in the product case, we can write
\[ i_{\Delta}^* Rj_* ( \mathscr{G}_1 \boxtimes \mathscr{G}_2 \otimes f^*(\mathscr{L}_{m_1} \boxtimes \mathscr{L}_{m_2}) ) \cong Rj_{\Delta,*} ( \mathscr{G}_1 \otimes f_1^* \mathscr{L}_{m_1} \otimes \mathscr{G}_2 \otimes f_2^* \mathscr{L}_{m_2}) \]
and the result follows by observing that the right hand side is also isomorphic to
\[ Rj_* ( \mathscr{G}_1 \otimes f_1^* \mathscr{L}_{m_1} ) \otimes Rj_* ( \mathscr{G}_2 \otimes f_2^* \mathscr{L}_{m_2} ) \]
where $Rj_*$ is here understood as the inclusion $X \setminus \{ s \} \hookrightarrow X$.
\end{proof}

We can immediately state a corollary by combining Gabber's observation with the main result of the previous section.

\begin{lemma}\label{nearbyDrinfeld}
Let $I$ and $J$ be finite sets, and let $I \to J$ be a map.  Let $\mathcal{G}^{\le \mu}$ be a collection of functors from $\Rep(\widehat{G}^{\bullet}) \to D^b_c(X^{\bullet})$ cofibered over $\FinS$ and filtered with respect to partial Frobenius maps $F_{\bullet}$.  Suppose also that $X$ is a Zariski dense open subset of $C_{\overline{\mathbb{F}_q}}$, where $C$ is a proper curve over a finite field.  Let $\overline{\eta}$ be a geometric generic point of $X$ and let $s$ be a special point of $C$.  Let $\Psi_I$ be iterated nearby cycles along $C^I$ with respect to the localization of $C$ at $s$.  Let $\Delta_{\zeta} \colon X^J \to X^I$ be the natural map, and let $\Psi_J$ be the same iterated nearby cycles.  Then, we have a natural identification
\begin{equation}
    \Psi_{J} \Delta_{\zeta}^* \varinjlim_{\mu} \mathcal{G}^{\le \mu}_{I, W} \cong \Psi_I \varinjlim_{\mu} \mathcal{G}^{\le \mu}_{I, W}
\end{equation}
which is equivariant with respect to the map $\Weil(\eta, \overline{\eta})^J \to \Weil(\eta, \overline{\eta})^I$.
\end{lemma}

We claim we also have the following map where $\Psi_{\Delta}$ denotes specialization from $\Delta(\overline{\eta})$ to $\Delta(s)$ over $C^I$.
\begin{lemma}\label{nearbycoalescence}
Consider the following cases:
\begin{enumerate}
    \item Nearby cycles is taken with respect to a geometric point $s$ over a closed point $p$ such that $G$ has parahoric reduction at $p$.
    \item Nearby cycles $\Psi$ is \textit{unipotent} nearby cycles and taken with respect to a geometric point $s$ over a point $p$.  $G$ is split on the generic fiber of the completion and has reduction according to the unipotent radical of the Iwahori.
\end{enumerate}

For $I = \{ 1, \dots, n \}$ and $W \cong W_1 \boxtimes \cdots \boxtimes W_n$, we have an isomorphism natural in $W$ and $V$.
\begin{equation}
    \Psi_{\Delta} \mathscr{F}_{\{ \Delta \} \cup J, (\bigotimes_{i \in I} W_i) \boxtimes V} \cong \Upsilon \mathscr{F}_{I \cup J, W \boxtimes V} \cong \Psi_1 \dots \Psi_n \mathscr{F}_{I \cup J, W \boxtimes V}.
\end{equation}
Both sides of the equation are sheaves over $\mathfrak{p}^{-1}(s \times (C \setminus \widehat{N})^J)$ for which suitable Schubert cells in $(\mathrm{Fl}_p)_{\overline{\mathbb{F}_q}} \times_{\overline{\mathbb{F}_q}} \Gr_{G,J}|_{(C \setminus \widehat{N})^J}$ are local models.
\end{lemma}

\begin{proof}
It suffices to check this theorem on sufficiently large open subsets, and the subsets given by Harder-Narasimhan truncation suffice for this purpose.

The local model theorem for shtukas, recalled above, says that for $\mu$ a fixed Harder-Narasimhan truncation of $G$-bundles (and therefore of $G$-shtukas) we can form the following diagram over $(C \setminus N)^{I \cup J}$:
\[ 
\begin{tikzcd}
\Gr_{G, I, W} & U \arrow[l, "u"] \arrow[r, "u'"] & \Cht_{G, \Xi, I, W}^{\le \mu}
\end{tikzcd}
\]
where the arrow on the right is an \'{e}tale \textit{cover}.  Since \'{e}tale maps are a fortiori smooth, we can consider
\begin{align*}
    u'^* \Psi_{\Delta} \mathscr{F}_{\{ \Delta \} \cup J, (\bigotimes_{i \in I} W_i) \boxtimes V} &\simeq \Psi_{\Delta} u'^* \mathscr{F}_{\{ \Delta \} \cup J, (\bigotimes_{i \in I} W_i) \boxtimes V} \\
    &\simeq \Psi_{\Delta} u^* \mathcal{S}_{\{ \Delta \} \cup J, (\bigotimes_{i \in I} W_i) \boxtimes V} \\
    &\simeq u^* \Psi_{\Delta} \mathcal{S}_{\{ \Delta \} \cup J, (\bigotimes_{i \in I} W_i) \boxtimes V} \\
    u'^* \Psi_1 \dots \Psi_n \mathscr{F}_{I \cup J, W \boxtimes V} &\simeq \Psi_1 \dots \Psi_n u'^* \mathscr{F}_{I \cup J, W \boxtimes V} \\
    &\simeq \Psi_1 \dots \Psi_n u^* \mathcal{S}_{I \cup J, W \boxtimes V} \\
    &\simeq u^* \Psi_1 \dots \Psi_n \mathcal{S}_{I \cup J, W \boxtimes V}
\end{align*}
We now show the following lemma:
\begin{lemma}\label{monoidalsheaves}
Under the assumptions of the above lemma, the correspondence
\[
\begin{tikzcd}
\Psi_{\Delta} \mathcal{S}_{\{ \Delta \} \cup J, (\bigotimes_{i \in I} W_i) \boxtimes V} & \Upsilon \mathcal{S}_{I \cup J, W \boxtimes V} \arrow[l] \arrow[r] & \Psi_1 \cdots \Psi_n \mathcal{S}_{I \cup J, W \boxtimes V}
\end{tikzcd}
\]
induces an isomorphism natural in $W$ and $V$,
\begin{equation}
    \Psi_{\Delta} \mathcal{S}_{\{ \Delta \} \cup J, (\bigotimes_{i \in I} W_i) \boxtimes V} \cong \Psi_1 \dots \Psi_n \mathcal{S}_{I \cup J, W \boxtimes V}.
\end{equation}
Both sides of the equation are sheaves over the preimage of $s \times (C \setminus (N \cup P))^J$ in the Beilinson-Drinfeld Grassmannian $\Gr_{G, \{ 1 \} \cup J}$.
\end{lemma}
Since the correspondence
\[
\begin{tikzcd}
\Psi_{\Delta} \mathscr{F}_{\{ \Delta \} \cup J, (\bigotimes_i W_i) \boxtimes V} & \Upsilon \mathscr{F}_{I \cup J, W \boxtimes V} \arrow[l] \arrow[r] & \Psi_1 \cdots \Psi_n \mathscr{F}_{I \cup J, W \boxtimes V}
\end{tikzcd}
\]
pulls back to an isomorphism by the lemma, it is an isomorphism to begin with.
\end{proof}

\begin{proof}[Proof of Lemma~\ref{monoidalsheaves}]
This statement is \'{e}tale local on the curve, so we can reduce to the case where $G$ has parahoric reduction only at a point $p \in C$.  We denote the complement by $C^\circ$.

Recall that $c$ denotes the convolution map on the Beilinson-Drinfeld Grassmannian.  By writing $\mathcal{S}_{I \cup J,W \boxtimes V}$ and $\mathcal{S}_{\{ \Delta \} \cup J, \otimes_i W_i \boxtimes V}$ as $c_! \mathcal{S}_{I,W} \tilde{\boxtimes} \mathcal{S}_{J,V}$ and $c_! \mathcal{S}_{\{ \Delta \},\otimes_i W_i} \tilde{\boxtimes} \mathcal{S}_{J,V}$ and applying proper base change, it suffices to show the isomorphisms
\[
\begin{tikzcd}
\Psi_{\Delta} \mathcal{S}_{\{ \Delta \}, (\bigotimes_{i \in I} W_i)} \tilde{\boxtimes} \mathcal{S}_{J,V} & \Upsilon \mathcal{S}_{I,W} \tilde{\boxtimes} \mathcal{S}_{J,V} \arrow[l] \arrow[r] & \Psi_1 \cdots \Psi_n \mathcal{S}_{I,W} \tilde{\boxtimes} \mathcal{S}_{J,V}
\end{tikzcd}
\]

By using the fact that the product of affine Grassmannians is an \'{e}tale local model for an iterated affine Grassmanian, it suffices to show that
\[
\begin{tikzcd}
\Psi_{\Delta} \mathcal{S}_{\{ \Delta \}, (\bigotimes_{i \in I} W_i)} \boxtimes \mathcal{S}_{J,V} & \Upsilon \mathcal{S}_{I,W} \boxtimes \mathcal{S}_{J,V} \arrow[l] \arrow[r] & \Psi_1 \cdots \Psi_n \mathcal{S}_{I,W} \boxtimes \mathcal{S}_{J,V}
\end{tikzcd}
\]
as sheaves on the product of affine Grassmannians.  The left side of the diagram can be identified with
\begin{equation}\label{Centralsheafcomputation}
Z_{\bigotimes_{i \in I} W_i} \boxtimes S_{J, V},
\end{equation}
where $Z_{\bigotimes_{i \in I} W_i}$ is a central sheaf following the work of Gaitsgory \cite{gaitsgory1999construction} and Zhu \cite{zhu2014coherence}.

We will prove the right hand side is equal to \ref{Centralsheafcomputation} by induction on $|I|$.  In the case $|I| = 1$, there is nothing to prove.  The case $|I| = 2$ is proved for the Iwahori case by Gaitsgory \cite{gaitsgory2004braiding} by using the $2$-dimensional iterated Beilinson-Drinfeld Grassmannian which is \'{e}tale locally a product of $1$-dimensional Beilinson-Drinfeld Grassmannians.  On the product we can apply the K\"unneth formula, and the small resolution allows us to transfer this result to the Grassmannian itself.  This result generalizes word for word to parahoric cases.

For the unipotent radical of the Iwahori, the relevant result is due to Bezrukavnikov \cite[Proposition~13]{bezrukavnikov2016two}.  We note that the central sheaves in that paper compute the case of unipotent nearby cycles at this level and in this setting Bezrukavnikov explains how to get around the failure of properness for the convolution product map.

Now the inductive step follows from the case $|I| = 2$, as we can use an iterated affine Grassmannian again to reduce to showing
\[ \Psi_0 ( Z_{\bigotimes_{i \in I} W_i} \boxtimes \mathcal{S}_{\{ 0 \}, U} \boxtimes \mathcal{S}_{J, V} ) \cong Z_{\bigotimes_{i \in I} W_i \otimes U} \boxtimes \mathcal{S}_{J, V} \]
which follows by rewriting the left hand side as
\[ \Psi_0 \Psi_{\Delta} ( \mathcal{S}_{\{ \Delta \},\bigotimes_{i \in I} W_i} \boxtimes \mathcal{S}_{\{0 \},U} \boxtimes S_{J, V} ). \]
\end{proof}

\begin{remark}
As noted by Cong Xue, the use of the local models theorem is not necessary, as one can also use the smoothness property of $\epsilon^{\Xi}$ instead, which has the advantage of being more canonical.
\end{remark}

Let $\Psi_J$ be the iterated nearby cycles $\Psi_{j_1} \cdots \Psi_{j_k}$ for $J = \{ j_1, \dots, j_k \}$.  Let $J_1$ and $J_2$ be identical, disjoint copies of $J$.  We use the following lemma, which generalizes Lemma~\ref{nearbycoalescence}.

\begin{lemma}\label{lem:monoidalJ}
For $(J_1, V_1)$ and $(J_2, V_2)$, let $(J, V)$ be defined by having a cocartesian map $(J_1 \cup J_2, V_1 \boxtimes V_2) \to (J, V)$ for $J_1 \cup J_2 \to J$ the map identifying both $J_1$ and $J_2$ with $J$.
\begin{equation}
    \Psi_{J} \mathscr{F}_{I \cup J, W \boxtimes V} \cong \Psi_{J_1} \Psi_{J_2} \mathscr{F}_{I \cup J_1 \cup J_2, W \boxtimes V_1 \boxtimes V_2}.
\end{equation}
Both sides of the equation are sheaves over $\mathfrak{p}^{-1}(s \times (C \setminus \widehat{N})^J) \cong (\mathrm{Fl}_p)_{\overline{\mathbb{F}_q}} \times_{\overline{\mathbb{F}_q}} \Gr_{G,J}|_{(C \setminus \widehat{N})^J}$.
\end{lemma}

\begin{proof}
We can naturally identify both sides to the sheaf
\[\Psi_{\Delta} \mathscr{F}_{I \cup \{ \Delta \}, W \boxtimes \left(\bigoplus \bigotimes V_{1,i} \otimes V_{2,i}\right)}.\]
The lemma can also be proved directly, following the proof of Lemma~\ref{nearbycoalescence}.
\end{proof}

The above proof seems to destroy information about the potential Weil group action, except on the diagonal, but we do not need it for that purpose: only to construct an inverse to the canonical map on nearby cycles, which is automatically Galois equivariant.

\begin{corollary}
Let $(J_1, V_1)$, $(J_2, V_2)$, and $(J, V)$ be as in the above lemma.  The isomorphisms in Lemma \ref{nearbycoalescence} and \ref{lem:monoidalJ} are compatible with the natural $!$-pushforward map.  In particular, 
\begin{equation}\label{diag:fusioncommutes}
    \begin{tikzcd}
    R \mathfrak{p}_! \Psi_{J} \mathscr{F}_{I \cup J, W \boxtimes V} \arrow[r] \arrow[d, "\simeq"] & \Psi_{J} R \mathfrak{p}_! \mathscr{F}_{I \cup J, W \boxtimes V} \arrow[d, "\simeq"] \\
    R \mathfrak{p}_! \Psi_{J_1} \Psi_{J_2} \mathscr{F}_{I \cup J_1 \cup J_2, W \boxtimes V_1 \boxtimes V_2} \arrow[r] & \Psi_{J_1} \Psi_{J_2} R \mathfrak{p}_! \mathscr{F}_{I \cup J_1 \cup J_2, W \boxtimes V_1 \boxtimes V_2}
    \end{tikzcd}
\end{equation}
\end{corollary}

\begin{proof}
This follows immediately from Proposition~\ref{prop:omnibusupsilon} (2), where the vertical arrow on the left is given by Lemma~\ref{lem:monoidalJ} and the vertical arrow on the right is given by Lemma~\ref{nearbyDrinfeld}.
\end{proof}

\section{Nearby cycles commute with pushforward via factorization structure}\label{sec:nearby}

In what follows, we often suppress mention of the group scheme $G$, which itself contains the information of the ramification and level structure along $\widehat{N}$.  We also suppress the lattice $\Xi$ in $\Cht_{G, \Xi, I \cup J, W \boxtimes V} =: \Cht_{I \cup J, W \boxtimes V}$.  We pick $p \in P \subset C$ and pick a geometric point $s$ above $p$.  Using this, we reduce to the local situation by taking a Henselian trait $S$ giving the specialization from a geometric generic point $\overline{\eta}$ of $C$ to $s$.

The following maps are defined from the geometric Satake functor and the coalescence of legs.
\[ \mathscr{F}_{\{1,2\},W \boxtimes 1} \to \mathscr{F}_{\{1,2\}, W \boxtimes (W^* \otimes W)}. \]
\[ \mathscr{F}_{\{2,3\},(W \otimes W^*) \boxtimes W} \to \mathscr{F}_{\{2,3\}, 1 \boxtimes W}. \]
The fact that the composition of these maps restricted to the diagonal is the identity is known as ``Zorro's lemma''. This fact will be used in Theorem \ref{thm:commutingnearbycycles}.  We will also use the following description of sheaves in the case when one leg carries the trivial representation.
\begin{lemma}\label{trivial}
The sheaf $\mathscr{F}_{I \cup \{ 1 \}, W \boxtimes 1}$ on $\Cht_{I \cup \{ 1 \}, W \boxtimes 1} |_{(C \setminus (N \cup P))^{I \cup \{ 1 \}}}$ is identified with the sheaf $\mathscr{F}_{I, W} |_{(C \setminus (N \cup P))^{I}} \boxtimes \overline{\mathbb{Q}_{\ell}}_{C \setminus (N \cup P)}$ under the identification
\begin{equation} \Cht_{I \cup \{ 1 \}, W \boxtimes 1} \cong \Cht_{I, W} \times (C \setminus N) \end{equation}
\end{lemma}

\begin{proof}
The identification of moduli spaces of shtukas is by sending $((x, x_1), \mathcal{E}, (\phi_I, \phi_1))$ to the pair of $(x, \mathcal{E}, \phi_I)$ and $x_1$, noting that $\phi_1$ must be the identity.  After restricting the locus, the identification of sheaves is by noting that this splitting is compatible with the corresponding splitting on the Beilinson-Drinfeld Grassmannian.
\end{proof}

Our main theorem shows that nearby cycles on the sheaves $\mathscr{F}$ commutes with compactly supported pushforward along $\mathfrak{p}$.  Let $I$ and $J$ be finite sets, and let $W \in \Rep(\widehat{G}^I)$ and $V \in \Rep(\widehat{G}^J)$ so we can form the sheaves $\mathscr{F}_{I \cup J, W \boxtimes V}$.  These sheaves live on shtukas over the power of the curve $C^{I \cup J}$, and we will take nearby cycles along directions $J$.  This will be an exact analogue of the result of \cite[Proposition~4.1.4, Proposition~3.2.1]{xue2020smoothness} but using nearby cycles instead of specialization maps.  In fact, this result can be seen to imply the smoothness result of loc. cit. over very special points by showing that vanishing cycles vanish on the sheaves over $(C \setminus \widehat{N})^{I}$.

\begin{theorem}\label{thm:commutingnearbycycles}
Consider the following cases:
\begin{enumerate}
    \item Nearby cycles is taken with respect to a geometric point $s$ over a closed point $p$ such that $G$ has parahoric reduction at $p$.
    \item Nearby cycles $\Psi$ is \textit{unipotent} nearby cycles and taken with respect to a geometric point $s$ over a point $p$.  $G$ is split on the generic fiber of the completion and has reduction according to the unipotent radical of the Iwahori.
\end{enumerate}

Then the natural map
\begin{equation} \mathrm{can} \colon R\mathfrak{p}_! \Psi_J \mathscr{F}_{I \cup J,W \boxtimes V} \to \Psi_J R\mathfrak{p}_! \mathscr{F}_{I \cup J, W \boxtimes V} \end{equation}
is an isomorphism.
\end{theorem}

\begin{proof}
We begin by constructing an inverse map.  Let $J_1, J_2, J_3$ be three identical, disjoint copies of $J$.
\begin{align*}
    \Psi_J R\mathfrak{p}_! \mathscr{F}_{I \cup J,W \boxtimes V} &\cong \Psi_{J_1} R\mathfrak{p}_! \Psi_{J_2} \mathscr{F}_{I \cup J_1 \cup J_2, W \boxtimes V \boxtimes 1} \\
    &\to \Psi_{J_1} R\mathfrak{p}_! \Psi_{J_2} \mathscr{F}_{I \cup J_1 \cup J_2, W \boxtimes V \boxtimes (V^* \otimes V)} \\
    &\cong \Psi_{J_1} R\mathfrak{p}_! \Psi_{J_2} \Psi_{J_3} \mathscr{F}_{I \cup J_1 \cup J_2 \cup J_3, W \boxtimes V \boxtimes V^* \boxtimes V} \\
    &\to \Psi_{J_1} \Psi_{J_2} R\mathfrak{p}_! \Psi_{J_3} \mathscr{F}_{I \cup J_1 \cup J_2 \cup J_3, W \boxtimes V \boxtimes V^* \boxtimes V} \\
    &\cong \Psi_{J_2} R\mathfrak{p}_! \Psi_{J_3} \mathscr{F}_{I \cup J_2 \cup J_3, W \boxtimes (V \otimes V^*) \boxtimes V} \\
    &\to \Psi_{J_2} R\mathfrak{p}_! \Psi_{J_3} \mathscr{F}_{I \cup J_2 \cup J_3, W \boxtimes 1 \boxtimes V} \\
    &\cong R\mathfrak{p}_! \Psi_{J} \mathscr{F}_{I \cup J, V \boxtimes W}
\end{align*}
We now justify that this gives a well-defined map.
\begin{enumerate}
    \item The first arrow is an isomorphism because of Lemma~\ref{trivial}.
    \item The second arrow follows by functoriality of $\mathrm{Rep}(G^I) \to \mathrm{Perv}(\Cht_{G, I})$.
    \item The third arrow exists and is an isomorphism by Lemma~\ref{lem:monoidalJ}.
    \item The fourth arrow is the canonical map $Rf_! \Psi \to \Psi Rf_!$.
    \item The fifth arrow is an isomorphism by Lemma~\ref{nearbyDrinfeld} and Lemma~\ref{nearby:shtuka}.
    \item The sixth arrow follows by functoriality.
    \item The seventh arrow is an isomorphism because of Lemma~\ref{trivial}.
\end{enumerate}

We now want to show that the map constructed is a left inverse of $\mathrm{can}$, and in particular $\mathrm{can}$ is injective.  This can be seen because the squares commute in Figure 5.1. %

\begin{figure}\label{bigzorro1}
\begin{tikzcd}
R\mathfrak{p}_! \Psi_{J_1} \Psi_{J_2} \mathscr{F}_{I \cup J_1 \cup J_2, W \boxtimes V \boxtimes 1} \arrow[rr] \arrow[dd]& & \Psi_{J_1} R\mathfrak{p}_! \Psi_{J_2} \mathscr{F}_{I \cup J_1 \cup J_2, W \boxtimes V \boxtimes 1} \arrow[dd] \\
& (1) & \\
R\mathfrak{p}_! \Psi_{J_1} \Psi_{J_2} \mathscr{F}_{I \cup J_1 \cup J_2, W \boxtimes V \boxtimes (V^* \otimes V)} \arrow[rr] \arrow[dd]& & \Psi_{J_1} R\mathfrak{p}_! \Psi_{J_2} \mathscr{F}_{I \cup J_1 \cup J_2, W \boxtimes V \boxtimes (V^* \otimes V)} \arrow[dd] \\
& (2) & \\
R\mathfrak{p}_! \Psi_{J_1} \Psi_{J_2} \Psi_{J_3} \mathscr{F}_{I \cup J_1 \cup J_2 \cup J_3, W \boxtimes V \boxtimes V^* \boxtimes V} \arrow[rr] \arrow[dd] \arrow[drr] & & \Psi_{J_1} R\mathfrak{p}_! \Psi_{J_2} \Psi_{J_3} \mathscr{F}_{I \cup J_1 \cup J_2 \cup J_3, W \boxtimes V \boxtimes V^* \boxtimes V} \arrow[d] \\
& (3) & \Psi_{J_1} \Psi_{J_2} R\mathfrak{p}_! \Psi_{J_3} \mathscr{F}_{I \cup J_1 \cup J_2 \cup J_3, W \boxtimes V \boxtimes V^* \boxtimes V} \arrow[d] \\
R\mathfrak{p}_! \Psi_{J_2} \Psi_{J_3} \mathscr{F}_{I \cup J_2 \cup J_3, W \boxtimes (V \otimes V^*) \boxtimes V} \arrow[rr] \arrow[dd] & & \Psi_{J_2} R\mathfrak{p}_! \Psi_{J_3} \mathscr{F}_{I \cup J_2 \cup J_3, W \boxtimes (V \otimes V^*) \boxtimes V} \arrow[dd] \\
& (4) & \\
R\mathfrak{p}_! \Psi_{J_2} \Psi_{J_3} \mathscr{F}_{I \cup J_2 \cup J_3, W \boxtimes 1 \boxtimes V} \arrow[rr]& & \Psi_{J_2} R\mathfrak{p}_! \Psi_{J_3} \mathscr{F}_{I \cup J_2 \cup J_3, W \boxtimes 1 \boxtimes V}
\end{tikzcd}
\caption{}
\end{figure}

In Figure 5.1, % \ref{bigzorro1},
the bottom row is an isomorphism because the sheaves are constant in the $J_2$ direction when the corresponding tuple of representations is trivial.  The left composition is the identity when specialized to the diagonal by Zorro's lemma.  The top arrow is the map $\mathrm{can}$.  The composed vertical line on the right, together with applications of Lemma~\ref{trivial}, is the map we constructed above as the left inverse of $\mathrm{can}$.  To show that this map is a left inverse, it suffices to check that all the squares in Figure 5.1 % \ref{bigzorro1}
commute.

\begin{enumerate}
    \item Square $(1)$ expresses the fact that $\mathrm{can}$ is a natural transformation, as the map on sheaves comes from the functoriality of $\mathscr{F}$ applied to the map $1 \to V^* \otimes V$.
    \item Square $(2)$ again expresses that $\mathrm{can}$ is a natural transformation, this time using the isomorphism in Lemma~\ref{lem:monoidalJ}.
    \item Square $(3)$ is an instance of Corollary~\ref{diag:fusioncommutes}.
    \item The bottom square $(4)$ follows from naturality of $\mathrm{can}$, with the map on sheaves coming from $\mathscr{F}$ applied to the map $V \otimes V^* \to 1$.
    \item The middle triangle composes two instances of $\mathrm{can}$.
\end{enumerate}

\begin{figure}\label{bigzorro2}
\begin{tikzcd}
\Psi_{J_1} R\mathfrak{p}_! \Psi_{J_2} \mathscr{F}_{I \cup J_1 \cup J_2, W \boxtimes V \boxtimes 1} \arrow[r] \arrow[d]& \Psi_{J_1} \Psi_{J_2} R\mathfrak{p}_! \mathscr{F}_{I \cup J_1 \cup J_2, W \boxtimes V \boxtimes 1} \arrow[d] \\
\Psi_{J_1} R\mathfrak{p}_! \Psi_{J_2} \mathscr{F}_{I \cup J_1 \cup J_2, W \boxtimes V \boxtimes (V^* \otimes V)} \arrow[r] \arrow[d]& \Psi_{J_1} \Psi_{J_2} R\mathfrak{p}_! \mathscr{F}_{I \cup J_1 \cup J_2, W \boxtimes V \boxtimes (V^* \otimes V)} \arrow[d] \\
\Psi_{J_1} R\mathfrak{p}_! \Psi_{J_2} \Psi_{J_3} \mathscr{F}_{I \cup J_1 \cup J_2 \cup J_3, W \boxtimes V \boxtimes V^* \boxtimes V} \arrow[r] \arrow[d]& \Psi_{J_1} \Psi_{J_2} \Psi_{J_3} R\mathfrak{p}_! \mathscr{F}_{I \cup J_1 \cup J_2 \cup J_3, W \boxtimes V \boxtimes V^* \boxtimes V} \arrow[dd] \\
\Psi_{J_1} \Psi_{J_2} R\mathfrak{p}_! \Psi_{J_3} \mathscr{F}_{I \cup J_1 \cup J_2 \cup J_3, W \boxtimes V \boxtimes V^* \boxtimes V} \arrow[d] \arrow[ur] & \\
\Psi_{J_2} R\mathfrak{p}_! \Psi_{J_3} \mathscr{F}_{I \cup J_2 \cup J_3, W \boxtimes (V \otimes V^*) \boxtimes V} \arrow[r] \arrow[d] & \Psi_{J_2} \Psi_{J_3} R\mathfrak{p}_! \mathscr{F}_{I \cup J_2 \cup J_3, W \boxtimes (V \otimes V^*) \boxtimes V} \arrow[d] \\
\Psi_{J_2} R\mathfrak{p}_! \Psi_{J_3} \mathscr{F}_{I \cup J_2 \cup J_3, W \boxtimes 1 \boxtimes V} \arrow[r] & \Psi_{J_2} \Psi_{J_3} R\mathfrak{p}_! \mathscr{F}_{I \cup J_2 \cup J_3, W \boxtimes 1 \boxtimes V}
\end{tikzcd}
\caption{}
\end{figure}

Now we want to show that it is a right inverse.  This can be seen because the squares commute in Figure 5.2. % \ref{bigzorro2}

The justifications for isomorphism of the top row and commutativity are the same as those in the previous paragraph, but in reverse order.
\end{proof}

\section{Consequences about the Galois action and Langlands parameters}\label{sec:Langlands}

Now that the key result about nearby cycles is available, we can use results about central sheaves from local geometric Langlands to draw a number of conclusions about the action of local Galois groups on cohomology.  This in turn will allow us to draw conclusions about the image of tame ramification in the Galois representation that Lafforgue associates to any automorphic form.

\begin{corollary}
Suppose $G$ has split parahoric reduction at a place $v$.  On the cohomology of shtukas, wild inertia acts trivially, and tame inertia acts unipotently.
\end{corollary}

\begin{proof}
By the result of Gaitsgory, generalized by Zhu \cite{zhu2014coherence}, nearby cycles on the affine Grassmannian are unipotent.  By the local models theorem, the action of the wild inertia on stalks of nearby cycles of $\mathscr{F}_{\{ 1 \}, W}$ is trivial and hence it acts trivially on the sheaf itself.  After pushing forward along $R\mathfrak{p}_!$, the wild inertia will still act trivially because the transformation $R\mathfrak{p}_! \Psi \to \Psi R\mathfrak{p}_!$ is Galois-equivariant, arising from proper base change theorem, and is an isomorphism in the case we are interested in.
\end{proof}

\begin{corollary}
Suppose $G$ is ramified at $v$ and after base change along a Galois extension $C' \to C$, for a point $v'$ lying over $v$, $G$ splits over $K'_{v'}$ and has parahoric reduction at $v'$.  Then wild inertia $I^1_{K'_{v'}}$ acts trivially on the cohomology of shtukas.  Moreover, if it acquires very special parahoric reduction, inertia of $K'_{v'}$ acts trivially.
\end{corollary}

\begin{proof}
This is the same as the above, except applying the work of Zhu \cite{zhu2011geometric} and later Richarz \cite{richarz2016affine}.  We note that the formation of the Beilinson-Drinfeld Grassmannian commutes with base change \cite[Lemma~3.2]{zhu2014coherence}, and we get trivial action of inertia on nearby cycles after base change.
\end{proof}

Let $K' / K$ be a finite Galois extension of the function field such that the generic fiber of $G$ splits, and let $C'$ be the normalization of the curve $C$ in $K'$.  Let $P'$ be the preimage in $C'$ of points $P$ of parahoric reduction.  Let $N'$ be the ``deep'' level structures $N \setminus P$.

For a point $p \in P$, we consider the completion of $K$ at $p$, $K_p$, and the extension $K'_{p'}$ for any point $p'$ over $p$ in $P'$.  Consider the local Weil groups $\Weil(K'_{p'}, \overline{K_{p'}})$ with its inertia subgroup $I_{{K'}_{p'}}$, wild inertia subgroup $I^1_{K'_{p'}}$, and the tame inertia quotient $I^t_{{K'}_{p'}} := I_{{K'}_{p'}} / I^1_{{K'}_{p'}}$.

Let $\Weil^{t}(C \setminus N', \overline{\eta})$ be the quotient of $\Weil(C \setminus \widehat{N}, \overline{\eta})$ by the normal subgroup generated by all wild inertia at all places in $P'$, that is, so that we have
\[
\begin{tikzcd}
\prod_{p' \in P'} I^1_{K'_{p'}} \arrow[r] & \Weil(C \setminus \widehat{N}, \overline{\eta}) \arrow[r] & \Weil^{t}(C \setminus N', \overline{\eta}) \arrow[r] & 1.
\end{tikzcd}
\]
The above theorems imply the following corollary.

\begin{corollary}
The action of $\Weil(C \setminus \widehat{N}, \overline{\eta})^I$ on $H_{I, W}$ factors through $\Weil^{t}(X \setminus N', \overline{\eta})^I$, and tame inertia subgroups along $P'$ act unipotently.
\end{corollary}

We now conclude with some consequences about the global Langlands correspondence.  Suppose $f$ is a cuspidal automorphic form associated to parahoric level structure at a point $p \in P \subset C$, and let $\rho \colon \Weil^{t}(C \setminus N') \to {}^{L}G$ be the associated Langlands parameter.  We note that this Langlands parameter restricts to a representation $\Weil^{t}(C' \setminus N') \to \widehat{G}$ and the tame inertia $I^t_{{K'}_{p'}}$ is a subset of the latter.  We would like to draw some conclusions about the image of tame generator of the local Galois group.

Let us briefly review some necessary facts about Langlands parameters that we will use.  Consider cuspidal automorphic forms for $G$ with finite order central character corresponding to the lattice $\Xi$ which is a fixed vector under an open compact subgroup corresponding to the level structure we have constructed via dilatation can be viewed as a function on $\Bun_G(\mathbb{F}_q) / \Xi$.  The corresponding Langlands parameters of cuspidal automorphic forms are characters of the algebra of excursion operators on $H_{\{ 0 \}, 1}^{\mathrm{cusp}}$.  Any such character $\nu$ will give a semisimple Galois representation $\rho$.  The Galois representation is compatible with the character of the excursion algebra in the sense that if $I$ is a finite set, $f \in \mathcal{O}(\widehat{G} \backslash \widehat{G}^I / \widehat{G})$ is a function and $(\gamma_i)_{i \in I}$ is a tuple of Galois elements, then the value of the character at the corresponding excursion operator $S_{I, f, \gamma}$ is equal to the value of $f$ at the corresponding tuple of $\widehat{G}^I$ under the Langlands parameter.  That is, \cite[Proposition~5.7]{lafforgue2014introduction}

\[ \nu(S_{I, f, (\gamma_i)_{i \in I}}) = f((\rho(\gamma_i))_{i \in I}). \]

For unipotence of the tame generator, we only need the case $I = \{ 1,2 \}$ and $(\gamma_1, \gamma_2) = (\gamma, 1)$ for $\gamma$ the tame generator.

\begin{theorem}\label{thm:unipotent}
Let $\rho : \Weil^t(C \setminus N', \overline{\eta}) \to {}^{L}G(\overline{\mathbb{Q}_\ell})$ be the Langlands parameter associated to a cuspidal automorphic form $f$ for $G(\mathbb{A}_K)$ with finite order central character and associated with a parahoric reduction by $P$ at a point $p \in C$.  Let $K'$ be the extension over which $G$ splits, $C'$ the normalization of $C$, and $P'$ the preimage in $C'$ of places in $C$ at which $G$ has parahoric reduction.

Then the image of the tame generator $\gamma \in I^t_{{K'}_{p'}} \subset \Weil^{t, P'}(C \setminus N, \overline{\eta}) \to {}^{L}G(\overline{\mathbb{Q}_{\ell}})$ is unipotent.
\end{theorem}

\begin{proof}
We can justify unipotence on the level of Galois representations up to conjugation, and we only need to use the fact that nearby cycles acts unipotently.  The image of the generator of tame ramification is unipotent if and only if its image in $\widehat{G} // \widehat{G}$ is the identity.  On the other hand, the unipotence of the action of $\gamma$ on cohomology implies that the excursion operator $\mathcal{S}_{\{ 1,2 \}, V, (\gamma,1)}$ acts unipotently for every $V$.  So for any character of the excursion algebra, $\mathcal{S}_{\{ 1,2 \}, V, (\gamma,1)} - 1$ is sent to $0$.  Running over all matrix coefficients up to conjugation, the image of $\gamma$ is $1 \in \widehat{G} // \widehat{G}$.
\end{proof}

At parahoric levels, we would like to know which unipotent conjugacy class contains the image of the tame generator.  This computation of monodromy is the analogue for shtukas of the computation of monodromy for automorphic sheaves in \cite[Section~4]{yun2016epipelagic}, see also \cite[Proposition~4.2.4]{yun2014motives}.

For the formulation of our next theorem, we introduce some notation.  For notational simplicity, we will assume that the generic fiber of $G$ splits.  Fix a pinning of $G$ and a fixed Iwahori $I$ in the loop group of the split group associated to $G$ which reduces modulo the uniformizer to a standard Borel $B$ over the base field.  A standard parahoric is a parahoric $P$ containing $I$.  For such a parahoric, we can form $P^+$, the pro-unipotent radical of $P$, and the corresponding Levi $M_P = P / P^+$, which is a reductive group.
Let $w_P$ be the longest element of the Weyl group of $M_P$.  For example, for the Iwahori $I$, $w_I$ is the identity.  We can view all the $w_P$ for different $P$ as elements of the affine Weyl group $\tilde{W}$.  Moreover, these are each contained in a unique two-sided cell, which we denote $c_P$.  In fact, they are Duflo involutions in the two-sided cells.  We now recall Lusztig's order-preserving bijection \cite{lusztig1989cells}
\[ \{ \text{two-sided cells in $\tilde{W}$} \} \simeq \{ \text{unipotent conjugacy classes in $\widehat{G}$} \}. \]

Let $c_P$ be the two-sided cell, and let $u_P$ the unipotent orbit corresponding to the two-sided cell $c_P$ in the Langlands dual group.  Let $\overline{u_P}$ be the closure of the unipotent orbit $u_P$ in $\widehat{G}$.  In general this is the closure of the unipotent orbit containing a dense open subset of the unipotent radical of the parabolic $\widehat{P} \subset \widehat{G}$.  We note that the two-sided cell containing the identity is the largest with respect to the order on the left hand side, and the orbit closure $\overline{u_I}$ is the unipotent cone $\mathcal{N}$ of the Langlands dual group.

We recall the main result of Bezrukavnikov \cite[Theorems 1 and 2]{bezrukavnikov2004tensor} in a format that is easy for us to use.  For any $w \in \tilde{W}$, let $L_w$ be the corresponding simple object in the Iwahori-Hecke category $\Perv(I \backslash LG / I)$.  To a two-sided cell $c$, Bezrukavnikov associates a tensor category $\mathcal{A}_c$ with tensor structure $\bullet$ given by truncated convolution.  For a Duflo involution $d \in c$, $L_d$ is idempotent under truncated convolution in $\mathcal{A}_c$.  Let $\mathcal{A}_d$ be the category generated by subquotients of $L_d \bullet Z(V)$ where $Z \colon \Rep(\widehat{G}) \to \Perv(I \backslash LG / I)$ is Gaitsgory's central functor into the Iwahori-Hecke category.  This gives us a tensor functor $\Rep(\widehat{G}) \to \mathcal{A}_d$ which is in fact a central functor.  This central functor inherits an automorphism from monodromy, and viewing this as a fiber functor and remembering only the monodromy automorphism of the fiber functor gives an element of $\widehat{G}$.

The result of Bezrukavnikov now implies that the conjugacy class of $\widehat{G}$ under Tannakian formalism is a unipotent conjugacy class corresponding to the two-sided cell $c$ containing the Duflo involution $d$.

Specializing to the case $d = w_P$, we have a smooth map of the Iwahori affine flag variety to the parahoric affine flag variety which intertwines the central sheaf functors and their monodromy.  That is, pulling back along this map sends central sheaves $\Psi(\mathcal{S}_V)$ for the parahoric to truncated convolutions $L_{w_P} \bullet Z(V)$ where $Z(V)$ is Gaitsgory's central sheaf functor for the Iwahori, since $L_{w_P}$ is just the constant sheaf in $\Perv(I \backslash P / I)$ for parahoric $P$ containing $I$.

\begin{theorem}\label{thm:twosided}
For a cuspidal automorphic form for a split group $G$ over $K$ with fixed vector under an open compact subgroup of $G(\mathbb{A}_K)$ with parahoric reduction at $p$ by a standard parahoric $P$, the image of the tame generator $\gamma \in I^t_{{K}_{p}} \subset \Weil^{t}(C \setminus N', \overline{\eta}) \to \widehat{G}(\overline{\mathbb{Q}_{\ell}})$ lives in the unipotent orbit closure $\overline{u_P}$ inside the Langlands dual group.
\end{theorem}

\begin{proof}
To prove the claim on two-sided cells, we will freely use the framed excursion algebra and related ideas from the paper of Lafforgue and Zhu \cite{lafforgue2018d}.  Consider the spectrum of the framed excursion algebra as an affine space of representations of $\Weil^{t}(C \setminus N, \overline{\eta})$, over which
\begin{equation} H^{\mathrm{cusp}}_{\{ 0 \}, \Reg} = \bigoplus_{V \in \mathrm{Rep}(\widehat{G})} H^{\mathrm{cusp}}_{\{0\}, V} \otimes V \end{equation}
is a module.  Choosing a generator $\gamma$ of tame ramification at $p$ gives a map of the spectrum of the space of representations to $\widehat{G}$.  This map is a restriction from all excursion operators to those involving just integer powers of the tame generator $\gamma$, including $1$.  We want to argue that the support of $H^{\mathrm{cusp}}_{\{ 0 \}, \Reg}$ on the excursion algebra is contained in the preimage of the unipotent orbit closure $\overline{u_P}$.  The Langlands parameters in V. Lafforgue's construction are the generalized eigenspaces given by the excursion algebra, and any character corresponds to a representation $\rho$ which must send the tame generator $\gamma$ to an element of $\overline{u_P}$.

Consider the ring of regular functions with $\overline{\mathbb{Q}_{\ell}}$ coefficients on $\widehat{G}$.  Let $\mathscr{J}$ be the ideal in this ring that defines the unipotent orbit closure $\overline{u_P}$.  It is generated by relations involving matrix coefficients that all elements in $\overline{u_P}$ satisfy.  Let $f$ be such an element of this ideal, that is, a polynomial vanishing on the Zariski closure of the unipotent orbit.  We first show that $f$ generates a relation that the monodromy of nearby cycles over $\Gr_G$ satisfies.  Then we use the local model to transport this relation to a relation satisfied by nearby cycles on $\Cht_G$.  Then, these relations will produce an ideal that exactly cuts out $\overline{u_P}$ in $\widehat{G}$ over which the framed moduli of Langlands parameters lives.

For any fixed regular function $f$ on $\widehat{G}$, by the algebraic Peter-Weyl theorem, $f$ can be written as a matrix coefficient $g \mapsto \langle v^*, g v \rangle$ for some representation $V$ (generally not irreducible) and vectors $v^* \in V^*, v \in V$.  For $V$ a $\widehat{G}$-representation, let $\underline{V}$ denote the underlying vector space with trivial $\widehat{G}$ action.  Following Lafforgue and Zhu we have a map of $\widehat{G}$ representations
\begin{equation} \theta \colon \Reg \otimes \underline{V} \to \Reg \otimes V \end{equation}
defined by $\theta(f \otimes v)(g) = f(g) g \cdot v$.

For $N_P \in u_P$ fixed, consider the map $F \colon \Reg \otimes \underline{V} \to \Reg \otimes \underline{V}$ of vector spaces (in fact it is a map in $\Rep(H_{w_P})$ in Bezrukavnikov's notation where $w_P$ is considered as a Duflo involution) defined by
\begin{equation} F(f \otimes v)(g) = f(g) g^{-1} N_P g \cdot v. \end{equation}
We note that $F$ implicitly depends on our choice of $N_P$.  This defines an endomorphism of $\Psi \mathcal{S}_{\Reg \otimes \underline{V}} \cong \Psi \mathcal{S}_{\Reg} \otimes \underline{V}$ considered as a map in the truncated convolution category.  Moreover, if $f \in \mathscr{J}$ represented by matrix coefficients pairing maps $x \colon 1 \to \underline{V}$ with $\xi \colon \underline{V} \to 1$, we know that the composition
\begin{equation}
\begin{tikzcd}
\Psi \mathcal{S}_{\{ 0 \}, \Reg} \arrow[r, "x"] & \Psi \mathcal{S}_{\{ 0 \}, \Reg \otimes \underline{V}} \arrow[r, "F"] & \Psi \mathcal{S}_{\{ 0 \}, \Reg \otimes \underline{V}} \arrow[r, "\xi"] & \Psi \mathcal{S}_{\{ 0 \}, \Reg}
\end{tikzcd}
\end{equation}
is $0$.  We note that there is a functor between equivariant sheaves on the parahoric affine flag variety to equivariant sheaves on the Iwahori affine flag variety sending $\Psi \mathcal{S}_{\{ 0 \}, V}$ to the object $L_{w_P} \bullet Z(V)$.  This functor restricts to a central functor from central sheaves on the parahoric affine flag variety to the truncated convolution category $\mathcal{A}_{w_P}$.  We also use $F$ to denote the map
\[ L_{w_P} \bullet Z(\Reg) \bullet Z(\underline{V}) \to L_{w_P} \bullet Z(\Reg) \bullet Z(\underline{V}) \]
under the equivalence of categories given by \cite[Theorem~1]{bezrukavnikov2004tensor}.

\begin{lemma}
For suitable $N_P \in u_P$, we have a commutative square in the truncated convolution category $\mathcal{A}_{w_P}$
\begin{equation}
    \begin{tikzcd}
    L_{w_P} \bullet Z(\Reg) \bullet Z(\underline{V}) \arrow[r, "F"] \arrow[d, "\theta"] & L_{w_P} \bullet Z(\Reg) \bullet Z(\underline{V}) \arrow[d, "\theta"] \\
    L_{w_P} \bullet Z(\Reg) \bullet Z(V) \arrow[r, "1 \bullet \mathfrak{M}_V"] & L_{w_P} \bullet Z(\Reg) \bullet Z(V)
    \end{tikzcd}
\end{equation}
where the vertical arrows are isomorphisms and the bottom row is the unipotent monodromy endomorphism coming from unipotent nearby cycles.
\end{lemma}

\begin{proof}
\[ \theta \circ F(f \otimes v)(g) = \theta(f \otimes g^{-1} N_P g v)(g) = f(g) N_P g \cdot v. \]
Under the identification $\mathcal{A}_{w_P} \cong \Rep(H_{w_P})$, the fact that the right hand side identifies with the composition of $1 * \mathfrak{M}_V \circ \theta$ follows by Theorem~1 of \cite{bezrukavnikov2004tensor}.
\end{proof}

We return to the proof of Theorem~\ref{thm:twosided}.  Using the above lemma, we build a diagram
\begin{equation}\label{eq:satakeexcursion}
    \begin{tikzcd}
    \Psi \mathcal{S}_{\{ 0 \}, \Reg} \arrow[r, "x"] & \Psi \mathcal{S}_{\{ 0 \}, \Reg \otimes \underline{V}} \arrow[d, "\theta"] \arrow[r, "F"] & \Psi \mathcal{S}_{\{ 0 \}, \Reg \otimes \underline{V}} \arrow[d, "\theta"] \arrow[r, "\xi"] & \mathcal{S}_{\{ 0 \}, \Reg} \\
    & \Delta^* \Psi_1 \Psi_0 \mathcal{S}_{\{ 0,1 \}, \Reg \boxtimes V} \arrow[r] & \Delta^* \Psi_1 \Psi_0 \mathcal{S}_{\{ 0,1 \}, \Reg \boxtimes V}
    \end{tikzcd}
\end{equation}
where the vertical arrows come from fusion of nearby cycles, the first and last arrow come from realizing the regular function $f$ as a matrix coefficient, and the lower horizontal arrow is the action of monodromy on the leg $1$.  The composition of the top row is $0$.

Such a relation is clearly preserved under pulling back under a smooth map, and we also know that pulling back \'{e}tale sheaves under a surjective map is a surjective functor.  Thus, we conclude that the composition
\begin{equation}
    \begin{tikzcd}
    \Psi \mathscr{F}_{\{ 0 \}, \Reg} \arrow[r, "x"] & \Psi \mathscr{F}_{\{ 0 \}, \Reg \otimes \underline{V}} \arrow[d] & \Psi \mathscr{F}_{\{ 0 \}, \Reg \otimes \underline{V}} \arrow[d] \arrow[r, "\xi"] & \mathscr{F}_{\{ 0 \}, \Reg} \\
    & \Delta^* \Psi_1 \Psi_0 \mathscr{F}_{\{ 0,1 \}, \Reg \boxtimes V} \arrow[r] & \Delta^* \Psi_1 \Psi_0 \mathscr{F}_{\{ 0,1 \}, \Reg \boxtimes V}
    \end{tikzcd}
\end{equation}
is $0$.  After pushing forward by $\mathfrak{p}_!$ and using Theorem~\ref{thm:commutingnearbycycles}, the composition becomes an element $F_{f, \gamma}$ in the framed excursion algebra, using notation in \cite{lafforgue2018d}.  Thus, we get relations in the framed excursion algebra
\[ F_{f, \gamma} = 0 \]
for every element $f$ of the ideal $\mathscr{J}$, which shows that the image of the tame generator is contained in the orbit closure $\overline{u_P}$.
\end{proof}

\begin{remark}
Since excursion operators can be constructed for the whole $H_{I, V}$ and not just the cuspidal part, and in the case of $G$ split, this cohomology is known to be finitely generated over the Hecke algebra \cite{xue2018finiteness}, the restriction to cuspidal automorphic forms in the above theorem can be relaxed.
\end{remark}

\printbibliography

\end{document}